\documentclass[11pt]{amsart}
\usepackage{latexsym,amsmath,amssymb,graphicx,xspace,color}

\newtheorem{mainthm}{Theorem}
\newtheorem*{mainconj}{Poset Curvature Conjecture}
\newtheorem{thm}{Theorem}[section]\newtheorem{lem}[thm]{Lemma}
\newtheorem{cor}[thm]{Corollary}\newtheorem{conj}[thm]{Conjecture}
\newtheorem{prop}[thm]{Proposition}
\theoremstyle{definition}\newtheorem{defn}[thm]{Definition}  
\newtheorem{rem}[thm]{Remark}\newtheorem{exmp}[thm]{Example}

\newcommand{\zero}{\mathbf{0}}\newcommand{\one}{\mathbf{1}} 
\newcommand{\mbb}[1]{\ensuremath{\mathbb{#1}}}
\newcommand{\N}{\mbb{N}}\newcommand{\Z}{\mbb{Z}} 
\newcommand{\R}{\mbb{R}}\newcommand{\F}{\mbb{F}}
\newcommand{\sph}{\mbb{S}} 

\newcommand{\tsc}[1]{\text{\textsc{#1}}}
\newcommand{\cat}{\tsc{CAT}}
\newcommand{\lk}{\tsc{lk}}

\newcommand{\sfw}[1]{\textsf{#1}\xspace}
\newcommand{\pe}{\sfw{PE}}\newcommand{\ps}{\sfw{PS}}

\newcommand{\sym}{\sfw{Sym}}\newcommand{\ortho}{\sfw{Ortho}} 
\newcommand{\script}[1]{\ensuremath{\mathcal{#1}}}

\newcommand{\order}[1]{\ensuremath{\lvert #1 \rvert}}
\newcommand{\bt}{\begin{tabular}}\newcommand{\et}{\end{tabular}}

\begin{document}
%%%%%%%%%%%%%%%%

\title{Braids, posets and orthoschemes}

\author[T.~Brady]{Tom Brady}
  \address{Mathematics\\
    Dublin City University\\
    Glasnevin, Dublin 9\\
    Ireland}
  \email{tom.brady@dcu.ie}

\author[J.~McCammond]{Jon McCammond}
  \address{Mathematics\\ 
    University of California\\ 
    Santa Barbara, CA 93106\\
    USA}
  \email{jon.mccammond@math.ucsb.edu} 
\date{\today}

\begin{abstract}
  In this article we study the curvature properties of the order
  complex of a graded poset under a metric that we call the
  ``orthoscheme metric''.  In addition to other results, we
  characterize which rank~$4$ posets have $\cat(0)$ orthoscheme
  complexes and by applying this theorem to standard posets and
  complexes associated with four-generator Artin groups, we are able
  to show that the $5$-string braid group is the fundamental group of
  a compact nonpositively curved space.
\end{abstract}

\maketitle

Barycentric subdivision subdivides an $n$-cube into isometric metric
simplices called \emph{orthoschemes}.  We use orthoschemes to turn the
order complex of a graded poset $P$ into a piecewise Euclidean complex
$K$ that we call its \emph{orthoscheme complex}.  Our goal is to
investigate the way that combinatorial properties of $P$ interact with
curvature properties of $K$.  More specifically, we focus on
combinatorial configurations in $P$ that we call \emph{spindles} and
conjecture that they are the only obstructions to $K$ being $\cat(0)$.

\begin{mainconj}\label{conj:poset}
  The orthoscheme complex of a bounded graded poset $P$ is $\cat(0)$
  iff $P$ has no short spindles.
\end{mainconj}

One way to view this conjecture is as an attempt to extend to a
broader context the flag condition that tests whether a cube complex
is $\cat(0)$.  We highlight this perspective in \S\ref{sec:low}.
Our main theorem establishes the conjecture for posets of low rank.

\begin{mainthm}\label{main:poset}
  The orthoscheme complex of a bounded graded poset $P$ of rank at
  most $4$ is $\cat(0)$ iff $P$ has no short spindles.
\end{mainthm}

Using Theorem~\ref{main:poset}, we prove that the $5$-string braid
group, also known as the Artin group of type $A_4$, is a $\cat(0)$
group.  More precisely, we prove the following.

\begin{mainthm}\label{main:artin}
  Let $K$ be the Eilenberg-MacLane space for a four-generator Artin
  group of finite type built from the corresponding poset of
  $W$-noncrossing partitions and endowed with the orthoscheme metric.
  When the group is of type $A_4$ or $B_4$, the complex $K$ is
  $\cat(0)$ and the group is a $\cat(0)$ group.  When the group is of
  type $D_4$, $F_4$ or $H_4$, the complex $K$ is not $\cat(0)$.
\end{mainthm}

The article is structured as follows.  The initial sections recall
basic results about posets, complexes and curvature, followed by
sections establishing the key properties of orthoschemes, orthoscheme
complexes and spindles.  The final sections prove our main results and
contain some concluding remarks.

%%%%%%%%%%%%%%%%%%%%%%%%%%%%%%%%%%%
\section{Posets}\label{sec:poset}
%%%%%%%%%%%%%%%%%%%%%%%%%%%%%%%%%%%

We begin with elementary definitions and results about posets.  For
additional background see \cite{Bj95} or \cite{EC1}.

\begin{defn}[Poset]
  A \emph{poset} is a set with a fixed implicit reflexive,
  anti-symmetric and transitive relation $\leq$.  A \emph{chain} is
  any totally ordered subset, subsets of chains are \emph{subchains}
  and a \emph{maximal chain} is one that is not a proper subchain of
  any other chain.  A chain with $n+1$ elements has \emph{length}~$n$
  and its elements can be labeled so that $x_0 < x_1 < \cdots < x_n$.
  A poset is \emph{bounded below}, \emph{bounded above}, or
  \emph{bounded} if it has a minimum element $\zero$, a maximum
  element $\one$, or both.  The elements $\zero$ and $\one$ are
  necessarily unique when they exist.  A poset has \emph{rank $n$} if
  it is bounded, every chain is a subchain of a maximal chain and all
  maximal chains have length~$n$.
\end{defn}

\begin{defn}[Interval]
  For $x\leq y$ in a poset $P$, the \emph{interval} between $x$ and
  $y$ is the restriction of the poset to those elements $z$ with $x
  \leq z \leq y$ and it is denoted by $P(x,y)$ or $P_{xy}$.  If every
  interval in $P$ has a rank, then $P$ is \emph{graded}.  Let $x$ be
  an element in a graded poset $P$.  When $P$ is bounded below, the
  \emph{rank of $x$} is the rank of the interval $P_{\zero x}$ and
  when $P$ is bounded above, the \emph{corank of $x$} is the rank of
  the interval $P_{x \one}$.  In general, if every interval in a poset
  $P$ has a particular property, we say $P$ \emph{locally} has that
  property.
\end{defn}

Note that a poset is bounded and graded iff it has rank~$n$ for some
$n$, and that the rank of an element $x$ in a bounded graded poset is
the same as the subscript $x$ receives when it is viewed as an element
of a maximal chain from $\zero$ to $\one$ whose elements are labelled
as described above.

\begin{defn}[Lattice]
  Let $x$ and $y$ be elements in a poset $P$.  An \emph{upper bound}
  for $x$ and $y$ is any element $z$ such that $x \leq z$ and $y \leq
  z$.  The minimal elements among the collection of upper bounds for
  $x$ and $y$ are called \emph{minimal upper bounds} of $x$ and $y$.
  When only one minimal upper bound exists, it is the \emph{join} of
  $x$ and $y$ and denoted $x \vee y$.  The definitions of \emph{lower
  bounds} and \emph{maximal lower bounds} of $x$ and $y$ are similar.
  When only one maximal lower bound exists, it is the \emph{meet} of
  $x$ and $y$ and denoted $x \wedge y$.  A poset in which $x \vee y$
  and $x \wedge y$ always exist is called a \emph{lattice}.
\end{defn}

For later use we define a particular configuration that is present in
every bounded graded poset that is not a lattice.

\begin{defn}[Bowtie]\label{def:bowtie}
  We say that a poset $P$ contains a \emph{bowtie} if there exist
  distinct elements $a$, $b$, $c$ and $d$ such that $a$ and $c$ are
  minimal upper bounds for $b$ and $d$ and $b$ and $d$ are maximal
  lower bounds for $a$ and $c$.  In particular, there is a zig-zag
  path from $a$ down to $b$ up to $c$ down to $d$ and back up to $a$.
  An example is shown in Figure~\ref{fig:non-lattice}.
\end{defn}

\begin{figure}
\includegraphics{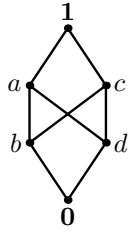}
\caption{A bounded graded poset that is not a lattice.\label{fig:non-lattice}}
\end{figure}

\begin{prop}[Lattice or bowtie]\label{prop:lattice-bowties}
  A bounded graded poset $P$ is a lattice iff $P$ contains no bowties.
\end{prop}

\begin{proof}
  If $P$ contains a bowtie, then $b$ and $d$ have no join and $P$ is
  not a lattice.  In the other direction, suppose $P$ is not a lattice
  because $x$ and $y$ have no join.  An upper bound exists because $P$
  is bounded, and a minimal upper bound exists because $P$ is graded.
  Thus $x$ and $y$ must have more than one minimal upper bound.  Let
  $a$ and $c$ be two such minimal upper bounds and note that $x$ and
  $y$ are lower bounds for $a$ and $c$.  If $b$ is a maximal lower
  bound of $a$ and $c$ satisfying $b \geq x$ and $d$ is a maximal
  lower bound of $a$ and $c$ satisfying $d \geq y$, then $a$, $b$,
  $c$, $d$ form a bowtie.  We know that $a$ and $c$ are minimal upper
  bounds of $b$ and $d$ and that $b$ and $d$ are distinct since either
  failure would create an upper bound of $x$ and $y$ that contradicts
  the minimality of $a$ and $c$.  When $x$ and $y$ have no meet, the
  proof is analogous.
\end{proof}

Posets can be used to construct simplicial complexes.

\begin{defn}[Order complex]
  The \emph{order complex} of a poset $P$ is a simplicial complex
  $\order{P}$ constructed as follows.  There is a vertex $v_x$ in
  $\order{P}$ for every $x \in P$, an edge $e_{xy}$ for all $x<y$ and
  more generally there is a $k$-simplex in $\order{P}$ with vertex set
  $\{v_{x_0}, v_{x_1}, \ldots, v_{x_k}\}$ for every finite chain $x_0
  < x_1 < \cdots < x_k$ in $P$.  When $P$ is bounded, $v_\zero$ and
  $v_{\one}$ are the \emph{endpoints} of $\order{P}$, and the edge
  $e_{\zero\one}$ connecting them is its \emph{diagonal}.
\end{defn}

The order complex of the poset shown in Figure~\ref{fig:non-lattice}
has $6$ vertices, $13$ edges, $12$ triangles and $4$ tetrahedra.
Since every maximal chain contains $\zero$ and $\one$, all four
tetrahedra contain the diagonal $e_{\zero\one}$.

\begin{prop}[Contractible]\label{prop:contractible}
  If a poset is bounded below or bounded above, then its order complex
  is contractible.
\end{prop}

\begin{proof}
  If $x$ is an extremum of $P$, then $\order{P}$ is a topological cone
  over the complex $\order{P \setminus \{x\}}$ with $v_x$ as the apex
  of the cone.
\end{proof}

%%%%%%%%%%%%%%%%%%%%%%%%%%%%%%%%%%%%%%%%%%
\section{Complexes}\label{sec:complex}
%%%%%%%%%%%%%%%%%%%%%%%%%%%%%%%%%%%%%%%%%%

Next we review the theory of piecewise Euclidean and piecewise
spherical cell complexes built out of Euclidean or spherical
polytopes, respectively.  For further background on polytopes see
\cite{Zi95} and for polytopal complexes see \cite{BrHa99}.

\begin{defn}[Euclidean polytope]
  A \emph{Euclidean polytope} is a bounded intersection of a finite
  collection of closed half-spaces of a Euclidean space, or,
  equivalently, it is the convex hull of a finite set of points.  A
  \emph{proper face} is a nonempty subset that lies in the boundary of
  a closed half-space containing the entire polytope.  Every proper
  face of a polytope is itself a polytope.  In addition there are two
  \emph{trivial faces}: the empty face $\emptyset$ and the polytope
  itself.  The \emph{interior} of a face is the collection of its
  points that do not belong to a proper subface, and every polytope is
  a disjoint union of the interiors of its faces.  The
  \emph{dimension} of a face is the dimension of the smallest affine
  subspace that containing it.  A $0$-dimensional face is a
  \emph{vertex} and a $1$-dimensional face is an \emph{edge}.
\end{defn}

\begin{defn}[\pe complex]
  A \emph{piecewise Euclidean complex} (or \emph{\pe complex}) is the
  regular cell complex that results when a disjoint union of Euclidean
  polytopes are glued together via isometric identifications of their
  faces.  For simplicity we usually insist that every polytope
  involved in the construction embeds into the quotient and that the
  intersection of any two polytopes be a face of each.  If there are
  only finitely many isometry types of polytopes used in the
  construction, we say it is a complex with \emph{finite shapes}.
\end{defn}

A basic result due to Bridson is that a \pe complex with finite shapes
is a geodesic metric space, i.e. the distance between two points in
the quotient is well-defined and achieved by a path of that length
connecting them.  This was a key result from Bridson's thesis
\cite{Br91} and is the main theorem of Chapter I.7 in \cite{BrHa99}.

The \pe complexes built out of cubes deserve special attention.

\begin{exmp}[Cube complexes]
  A \emph{cube complex} is a (connected) \pe complex $K$ in which
  every cell used in its construction is isometric to a metric cube of
  some dimension.  Although it is traditional to use cubes with unit
  length sides, this is not strictly necessary.  The fact that faces
  of $K$ are identified by isometries, together with connectivity,
  does, however, imply that every edge has the same length and thus
  $K$ is a rescaled version of a unit cube complex.
\end{exmp}

In the same way that every poset has an associated cell complex, every
regular cell complex has an associated poset.

\begin{defn}[Face posets]
  Every regular cell complex $K$, such as a \pe complex, has an
  associated \emph{face poset} $P$ with one element $x_\sigma$ for
  each cell $\sigma$ in $K$ ordered by inclusion, that is $x_\sigma
  \leq x_\tau$ iff $\sigma \subset \tau$ in $K$.  As is well-known,
  the operations of taking the face poset of a cell complex and
  constructing the order complex of a poset are nearly but not quite
  inverses of each other.  More specifically, the the order complex of
  the face poset of a regular cell complex is a topological space
  homeomorphic to original cell complex but with a different cell
  structure.  The new cells are obtained by barycentrically
  subdividing the old cells.
\end{defn}

\begin{defn}[Euclidean product]
  Let $K$ and $L$ be \pe complexes.  The metric on the product complex
  $K\times L$ is defined in the natural way: if the distance in $K$
  between $x$ and $x'$ is $d_K$ and the distance in $L$ between $y$
  and $y'$ is $d_L$, then the distance between $(x,y)$ and $(x',y')$
  is $\sqrt{(d_K)^2+(d_L)^2}$.  It is itself a \pe complex built out
  of products of polytopes.  More precisely, if $\sigma$ is a nonempty
  cell in $K$ and $\tau$ is a nonempty cell in $L$, then there is a
  polytope $\sigma \times \tau$ in $K \times L$.  Conversely, every
  nonempty cell in $K\times L$ is a polytope of this form.
\end{defn}

Euclidean polytopes and \pe complexes have spherical analogs.

\begin{defn}[Spherical polytope]
  A \emph{spherical polytope} is an intersection of a finite
  collection of closed hemispheres in $\sph^n$, or, equivalently, the
  convex hull of a finite set of points in $\sph^n$.  In both cases
  there is an additional requirement that the intersection or convex
  hull be contained in some open hemisphere of $\sph^n$.  This avoids
  antipodal points and the non-uniqueness of the geodesics connecting
  them.  With closed hemispheres replacing closed half-spaces and
  lower dimensional unit subspheres replacing affine subspaces, the
  other definitions are unchanged.
\end{defn}

\begin{defn}[\ps complex]
  A \emph{piecewise spherical complex} (or \emph{\ps complex}) is the
  regular cell complex that results when a disjoint union of spherical
  polytopes are glued together via isometric identifications of their
  faces.  As above we usually insist that each polytope embeds into
  the quotient and that they intersect along faces.  As above, so long
  as the complex has finite shapes, the result is a geodesic metric
  space.
\end{defn}

\begin{defn}[Vertex links]\label{def:vertex-lks}
  Let $v$ be a vertex of a Euclidean polytope $\sigma$.  The
  \emph{link of $v$ in $\sigma$} is the set of directions that remain
  in $\sigma$.  More precisely, the \emph{vertex link} $\lk(v,\sigma)$
  is the spherical polytope of unit vectors $u$ such that $v+\epsilon
  u$ is in $\sigma$ for some $\epsilon>0$.  More generally, let $v$ be
  a vertex of a \pe complex $K$.  The \emph{link of $v$ in $K$},
  denoted $\lk(v,K)$ is obtained by gluing together the spherical
  polytopes $\lk(v,\sigma)$ where $\sigma$ is a Euclidean polytope in
  $K$ with $v$ as a vertex.  The intuition is that $\lk(v,K)$ is a
  rescaled version of the boundary of an $\epsilon$-neighborhood of
  $v$ in $K$.
\end{defn}

A vertex link of a Euclidean polytope is a spherical polytope, a
vertex link of a \pe complex is a \ps complex, and a vertex link of a
cube complex is a simplicial complex.  The converse is also true in
the sense that every spherical polytope is a vertex link of some
Euclidean polytope, every \ps complex is a vertex link of some \pe
complex, and every simplicial complex is a vertex link of some cube
complex.

\begin{defn}[Spherical joins]
  Given \ps complexes $K$ and $L$, we define a new \ps complex $K\ast
  L$ that is the spherical analog of Euclidean product.  As remarked
  above, there is a \pe complex $K'$ with a vertex $v$ such that
  $K=\lk(v,K')$ and a \pe complex $L'$ with a vertex $w$ such that
  $L=\lk(w,L')$.  We define $K\ast L$ to be the link of $(v,w)$ in
  $K'\times L'$.  The cell structure of $K \ast L$ as a \ps complex is
  built from spherical joins of cells of $K$ with $L$.  In particular,
  each spherical polytope in $K\ast L$ is $\sigma \ast \tau$ where
  $\sigma$ is a cell of $K$ and $\tau$ is a cell of $L$.  One
  difference in the spherical context is that the empty face plays an
  important role.  There are cells in $K \ast L$ of the form $\sigma
  \ast \emptyset$ and $\emptyset \ast \tau$ that fit together to form
  a copy of $K$ and a copy of $L$, respectively.  What is going on is
  that the link of $(v,w)$ in $K' \times \{w\}=K'$ is $K \ast
  \emptyset=K$ and the link of $(v,w)$ in $\{v\} \times L'=L'$ is
  $\emptyset \ast L = L$.  The spherical join $K \ast L$ can
  alternatively be defined as the smallest \ps complex containing a
  copy of $K$ and a copy of $L$ such that every point of $K$ is
  distance $\frac{\pi}{2}$ from every point of $L$
  \cite[p.63]{BrHa99}.  Spherical join is a commutative and
  associative operation on \ps complexes with the empty complex
  $\emptyset$ as its identity.
\end{defn}

Next we extend the notion of a vertex link to the link of a face of a
polytope and the link of a cell in a \pe complex.

\begin{defn}[Face links]
  Let $x$ be a point in an $n$-dimensional Euclidean polytope
  $\sigma$, let $\tau$ be the unique face of $\sigma$ that contains
  $x$ in its interior, let $k$ be the dimension of $\tau$, and define
  $\lk(x,\sigma)$ as in Definition~\ref{def:vertex-lks}.  If $x$ is not a
  vertex, then $\lk(x,\sigma)$ is not a spherical polytope.  In fact,
  $\lk(x,\sigma) \supset \lk(x,\tau) = \sph^{k-1}$ which contains
  antipodal points since $k>0$.  To remedy this situation we note that
  $\lk(x,\sigma) = \lk(x,\tau) \ast \lk(\tau,\sigma)$ where the latter
  is a spherical polytope defined as follows.  The \emph{link of
    $\tau$ in $\sigma$} is the set of directions perpendicular to
  $\tau$ that remain in $\sigma$.  More precisely, the \emph{face
    link} $\lk(\tau,\sigma)$ is the spherical polytope of unit vectors
  $u$ perpendicular to the affine hull of $\tau$ such that for any $x$
  in the interior of $\tau$, $x+\epsilon u$ is in $\sigma$ for some
  $\epsilon>0$.  More generally, let $\tau$ be a cell of a \pe complex
  $K$.  The \emph{link of $\tau$ in $K$}, denoted $\lk(\tau,K)$, is
  obtained by gluing together the spherical polytopes
  $\lk(\tau,\sigma)$ where $\sigma$ is a Euclidean polytope in $K$
  with $\tau$ as a face.  As an illustration, if $x$ is a point in an
  edge $\tau$ of a tetrahedron $\sigma$, then $\lk(\tau,\sigma)$ is a
  spherical arc whose length is the size of the dihedral angle between
  the triangles containing $\tau$, whereas $\lk(x,\tau) = \sph^0$ and
  $\lk(x,\sigma) = \lk(x,\tau) \ast \lk(\tau,\sigma) = \sph^0 \ast
  \lk(\tau,\sigma)$ is a lune of the $2$-sphere.  More generally, if
  $x$ is a point in a \pe complex $K$, $\tau$ is the unique cell of
  $K$ containing $x$ in its interior, and $k$ is the dimension of
  $\tau$, then, viewing $\lk(x,K)$ as a rescaling of the boundary of
  an $\epsilon$-neighborhood of $x$ in $K$, we have that $\lk(x,K) =
  \lk(x,\tau) \ast \lk(\tau,K) = \sph^{k-1} \ast \lk(\tau,K)$.
\end{defn}

\begin{defn}[Links in \ps complexes]
  Let $\sigma$ be a cell in a \ps complex $K$.  To define
  $\lk(\sigma,K)$ we find a \pe complex $K'$ with vertex $v$ such that
  $K=\lk(v,K')$ and then identify the unique cell $\sigma'$ in $K'$
  such that $\sigma=\lk(v,\sigma')$.  We then define the \ps complex
  $\lk(\sigma,K)$ as the \ps complex $\lk(\sigma',K')$.
\end{defn}

We conclude by recording some elementary properties of links and
joins.

\begin{prop}[Links of links]\label{prop:lk-lk}
  If $\sigma \subset \sigma'$ are cells in a \pe or \ps complex $K$
  then there is a cell $\tau$ in $L=\lk(\sigma,K)$ such that
  $\lk(\tau,L) = \lk(\sigma',K)$.  Moreover, the link of every cell
  $\tau$ in $\lk(\sigma,K)$ arises in this way.  In other words, a
  link of a cell in the link of a cell is a link of a larger cell in
  the original complex.
\end{prop}

\begin{prop}[Links of joins]\label{prop:link-join}
  Let $K$ and $L$ be \ps complexes with cells $\sigma$ and $\tau$
  respectively.  If $K' = \lk(\sigma,K)$ and $L' = \lk(\tau,L)$, then
  $K' \ast L$, $K \ast L'$ and $K' \ast L'$ are links of cells in $K
  \ast L$.  Moreover, every link of a cell in $K \ast L$ is of one of
  these three types.
\end{prop}

%%%%%%%%%%%%%%%%%%%%%%%%%%%%%%%%%%%%%%%%%%
\section{Curvature}\label{sec:curvature}
%%%%%%%%%%%%%%%%%%%%%%%%%%%%%%%%%%%%%%%%%%

As a final bit of background, we review curvature conditions such as
$\cat(0)$ and $\cat(1)$.  In general these terms are defined by
requiring that certain geodesic triangles be ``thinner'' than
comparison triangles in $\R^2$ or $\sph^2$, but because we are just
reviewing $\cat(0)$ \pe complexes and $\cat(1)$ \ps complexes,
alternate definitions are available that only involve the existence of
short geodesic loops in links of cells.

\begin{defn}[Geodesics and geodesic loops]\label{def:geos}
  A \emph{geodesic} in a metric space is an isometric embedding of a
  metric interval and a \emph{geodesic loop} is an isometric embedding
  of a metric circle.  A \emph{local geodesic} and \emph{local
    geodesic loop} are weaker notions that only require the image
  curves be locally length minimizing.  For example, a path more than
  halfway along the equator of a $2$-sphere is a local geodesic but
  not a geodesic and a loop that travels around the equator twice is a
  local geodesic loop but not a geodesic loop.  A loop in a \ps
  complex of length less than $2\pi$ is called \emph{short} and a \ps
  complex that contains no short local geodesic loops is called
  \emph{large}.
\end{defn}

\begin{defn}[Curvature conditions]\label{def:curv}
  If $K$ is a \pe complex with finite shapes and the link of every
  cell in $K$ is large, then $K$ is \emph{nonpositively curved} or
  \emph{locally $\cat(0)$}.  If, in addition, $K$ is connected and
  simply connected, then $K$ is \emph{$\cat(0)$}.  As a consequence of
  the general theory of $\cat(0)$ spaces, such a complex $K$ is
  contractible.  If $K$ is a \ps complex and the link of every cell in
  $K$ is large, then $K$ is \emph{locally $\cat(1)$}.  If, in
  addition, $K$ itself is large, then $K$ is \emph{$\cat(1)$}.
\end{defn}

It follows from the definitions and Proposition~\ref{prop:lk-lk} that a \pe
complex is nonpositively curved iff its vertex links are $\cat(1)$.
Moreover, in the same way that every \ps complex is a vertex link of a
\pe complex, every $\cat(1)$ \ps complex is the vertex link of a
$\cat(0)$ \pe complex.  

A standard example where the $\cat(0)$ condition is easy to check is
in a cube complex.  The link of a cell in a cube complex is a \ps
simplicial complex built out of all-right spherical simplices,
i.e. spherical simplices in which every edge has length
$\frac{\pi}{2}$.  To check whether a cube complex is $\cat(0)$ it is
sufficient to check whether its vertex links satisfy the purely
combinatorial condition of being flag complexes.

\begin{defn}[Flag complexes]\label{def:flag}
  A simplicial complex contains an \emph{empty triangle} if there are
  three vertices pairwise joined by edges but no triangle with these
  three vertices as its corners.  More generally, a simplicial complex
  $K$ contains an \emph{empty simplex} if for some $n>1$, $K$ contains
  the boundary of an $n$-simplex but no $n$-simplex with the same
  vertex set.  A \emph{flag complex} is a simplicial complex with no
  empty simplices.
\end{defn}

\begin{prop}[$\cat(0)$ cube complexes]\label{prop:cube-curvature}
  A cube complex $K$ is $\cat(0)$ iff $\lk(v,K)$ is a flag complex for
  every vertex $v$ which is true iff $\lk(\sigma,K)$ has no empty
  triangles for every cell $\sigma$.
\end{prop}

If a \ps complex $K$ is locally $\cat(1)$ but not $\cat(1)$ (i.e. the
links of $K$ are large, but $K$ itself is not large) then we say $K$
is \emph{not quite $\cat(1)$}.  In \cite{Bo95} Brian Bowditch
characterized not quite $\cat(1)$ spaces using the notion of a
shrinkable loop.

\begin{defn}[Shrinkable loop]\label{def:shrinkable}
  Let $\gamma$ be a rectifiable loop of finite length in a metric
  space.  We say that $\gamma$ is \emph{shrinkable} if there exists a
  continuous deformation from $\gamma$ to a loop $\gamma'$ that
  proceeds through rectifiable curves of finite length in such a way
  that the lengths of the intermediate curves are nonincreasing and
  the length of $\gamma'$ is strictly less than the length of
  $\gamma$.  If $\gamma$ is not shrinkable, it is \emph{unshrinkable}.
\end{defn}

The following is a special case of the general results proved in
\cite{Bo95}.

\begin{lem}[Not quite $\cat(1)$]\label{lem:bowditch}
  If $K$ is a locally $\cat(1)$ \ps complex with finite shapes, then
  the following are equivalent:
  \begin{enumerate}
  \item[1.] $K$ is not quite $\cat(1)$,
  \item[2.] $K$ contains a short geodesic loop,
  \item[3.] $K$ contains a short local geodesic loop, and
  \item[4.] $K$ contains a short unshrinkable loop.
  \end{enumerate}
\end{lem}

This is an extremely useful lemma since it is sometimes easier to
establish that every loop in a space is shrinkable than it is to show
that it does not contain a short local geodesic.  Sometimes, for
example, a single homotopy shrinks every curve simultaneously.

\begin{defn}[Monotonic contraction]\label{def:mono-contract}
  Let $X$ be a metric space and let $H: X\times [0,1] \to X$ be a
  homotopy contracting $X$ to a point (i.e. $H_0$ is the identity map
  and $H_1$ is a constant map).  We say $H$ is a \emph{monotonic
    contraction} if $H$ simultaneously and monotonically shrinks every
  rectifiable curve in $X$ to a point.
\end{defn}

An example of a monotonic contraction is straightline homotopy from
the identity map on $\R^n$ to a constant map.  A spherical version of
monotonic contraction is needed in \S\ref{sec:orthoscheme}.

\begin{defn}[Hemispheric contraction]\label{def:hemi-contract}
  Let $u$ be a point in $\sph^n$ and let $X$ be a hemisphere of
  $\sph^n$ with $u$ as its pole, i.e. the ball of radius
  $\frac{\pi}{2}$ around $u$.  Every point $v$ in $X$ lies on a unique
  geodesic connecting $v$ to $u$ and we can define a homotopy that
  moves $v$ to $u$ at a constant speed so that at time $t$ it has
  traveled $t$ of the distance along this geodesic.  This
  \emph{hemispheric contraction} to $u$ is a monotonic contraction in
  the sense defined above.
\end{defn}

Although Lemma~\ref{lem:bowditch} only applies to locally $\cat(1)$
\ps complexes, such contexts can always be found when curvature
conditions fail.

\begin{prop}[Curvature and links]\label{prop:curv-lks}
  Let $K$ be a connected and simply-connected \pe complex with finite
  shapes.  If $K$ is not $\cat(0)$ then there is a cell $\sigma$ in
  $K$ such that $\lk(\sigma,K)$ is not quite $\cat(1)$.  Similarly, if
  $K$ is a \ps complex that is not $\cat(1)$ then either $K$ itself is
  not quite $\cat(1)$ or there is a cell $\sigma$ such that
  $\lk(\sigma,K)$ is not quite $\cat(1)$.
\end{prop}

\begin{proof}
  Let $\script{S}$ be the set of the cells $\sigma$ in $K$ such that
  $\lk(\sigma,K)$ is not $\cat(1)$ and order them by inclusion.
  Unless $K$ itself is a not quite $\cat(1)$ \ps complex, the set
  $\script{S}$ is not empty.  Moreover, because $K$ has finite shapes,
  $K$ is finite dimensional, and $\script{S}$ has maximal elements.
  If $\sigma$ is such a maximal element, then maximality combined with
  Proposition~\ref{prop:lk-lk} shows that $\lk(\sigma,K)$ is locally
  $\cat(1)$.  Since $\sigma$ is in $\script{S}$, $\lk(\sigma,K)$ is
  not $\cat(1)$.  Thus $\lk(\sigma,K)$ is not quite $\cat(1)$.
\end{proof}

To show that a \pe or \ps complex is not $\cat(0)$ or $\cat(1)$ it is
convenient to be able to construct and detect local geodesic loops.
We do this using piecewise geodesics.

\begin{defn}[Piecewise geodesics]
  Let $K$ be a \pe or \ps complex and let $(x_0,x_1,\ldots,x_n)$ be a
  sequence of points in $K$ such that for each $i$, $x_{i-1}$ and
  $x_i$ belong to a unique minimal common cell of $K$ and $x_0=x_n$.
  The \emph{piecewise geodesic loop} defined by this list is the
  concatenation of the unique geodesics from $x_{i-1}$ to $x_i$ in the
  (unique) minimal common cell containing them.  The points $x_i$ are
  its \emph{transition points}.  Piecewise geodesics are local
  geodesics iff they are locally geodesic at its transition points.
  This can be determined by examining two \emph{transition vectors}:
  the unit tangent vector at $x_i$ for the geodesic connecting $x_i$
  to $x_{i+1}$ and the unit tangent vector at $x_i$ for the geodesic
  from $x_i$ to $x_{i-1}$ (traversed in reverse).  These correspond to
  two points in $\lk(x_i,K)$.  We say that the transition points are
  \emph{far apart} if the distance between them is at least $\pi$
  inside $\lk(x_i,K)$.  Finally, a piecewise geodesic loop $\gamma$ in
  a \ps or \pe complex $K$ is a \emph{local geodesic} iff at every
  transition point $x$, the transition vectors are far apart in
  $\lk(x,K)$.
\end{defn}

We conclude by relating curvature, links and spherical joins.

\begin{prop}[$\cat(1)$ and joins]\label{prop:cat1-joins}
  Let $K$ and $L$ be \ps complexes.
  \begin{enumerate}
  \item[1.] $K \ast L$ is $\cat(1)$ iff $K$ and $L$ are $\cat(1)$.
  \item[2.] If $K \ast L$ is locally $\cat(1)$ then $K$ and $L$ are
    locally $\cat(1)$.
  \item[3.] If $K \ast L$ is not quite $\cat(1)$ then $K$ or $L$ is
    not quite $\cat(1)$.
  \end{enumerate}
  Similar assertions hold for spherical joins of the form $K_1 \ast
  K_2 \ast \cdots \ast K_n$.
\end{prop}

\begin{proof}
  The first part is Corollary II.3.15 in \cite{BrHa99}.  For the
  second assertion suppose $K \ast L$ is locally $\cat(1)$ and let
  $K'$ be a link of $K$.  Since $K' \ast L$ is a link of $K \ast L$
  (Proposition~\ref{prop:link-join}), it is $\cat(1)$ by assumption.  But
  then $K'$ is $\cat(1)$ by the part $1$ and $K$ is locally $\cat(1)$.
  The third part merely combines the first two and extending to
  multiple joins is an easy induction.
\end{proof}

%%%%%%%%%%%%%%%%%%%%%%%%%%%%%%%%%%%%%%%%%%%%%%%%%%
\section{Orthoschemes}\label{sec:orthoscheme}
%%%%%%%%%%%%%%%%%%%%%%%%%%%%%%%%%%%%%%%%%%%%%%%%%%

In this section we introduce the shapes that H.S.M.~Coxeter called
``orthoschemes'' \cite{Cox91} and our main goal is to establish that
face links in orthoschemes decompose into simple shapes
(Corollary~\ref{cor:ortho-links}).  Roughly speaking an orthoscheme is the
convex hull of a piecewise linear path that proceeds along mutually
orthogonal lines.

\begin{defn}[Orthoschemes]
  Let $(v_0, v_1, \ldots, v_n)$ be an ordered list of $n+1$ points in
  $\R^n$ and for each $i\in [n]$ let $u_i$ be the vector from
  $v_{i-1}$ to $v_i$.  If the vectors $\{u_i\}$ form an orthogonal
  basis of $\R^n$ then the convex hull of the points $\{v_i\}$ is a
  metric $n$-simplex called an \emph{$n$-orthoscheme} that we denote
  $\ortho(v_0,\ldots,v_n)$.  If the vectors $\{u_i\}$ form an
  orthonormal basis, then it is a \emph{standard $n$-orthoscheme}.  It
  follows easily from the definition that every face of an orthoscheme
  (formed by selecting a subset of its vertices) is itself an
  orthoscheme, although not all faces of a standard orthoscheme are
  standard.  Faces defined by consecutive vertices are of particular
  interest and we use $O(i,j)$ to denote the face
  $\ortho(v_i,\ldots,v_j)$ of $O=\ortho(v_0,\ldots,v_n)$.  The points
  $v_i$ are the \emph{vertices} of the orthoscheme, $v_0$ and $v_n$
  are its \emph{endpoints} and the edge connecting $v_0$ and $v_n$ is
  its \emph{diagonal}.  A $3$-orthoscheme is shown in
  Figure~\ref{fig:3-ortho}.
\end{defn}

\begin{figure}
  \includegraphics{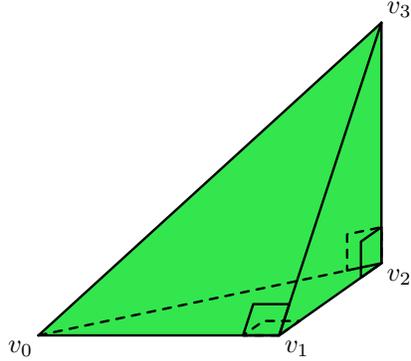}
  \caption{A $3$-dimensional orthoscheme.\label{fig:3-ortho}}
\end{figure}

For later use we define a contraction of an endpoint link of an
orthoscheme.

\begin{prop}[Endpoint contraction]\label{prop:endpt-contract}
  If $v$ is an endpoint of an orthoscheme $O$ and $u$ is the unit
  vector pointing from $v$ along its diagonal, then hemispheric
  contraction to $u$ monotonically contracts $\lk(v,O)$.
\end{prop}

\begin{proof}
  Let $O=\ortho(v_0,\ldots,v_n)$ and let $v=v_0$.  Without loss of
  generality arrange $O$ in $\R^n$ so that $v_0$ is the origin and
  each vector $v_i - v_{i-1}$ is a positive scalar multiple of a
  standard basis vector.  In this coordinate system all of $O$ lies in
  the nonnegative orthant and the coordinates of $v_n$ are strictly
  positive.  In particular the convex spherical polytope $\lk(v,O)$ is
  contained in an open hemisphere of $\sph^{n-1}$ with $u$ as its
  pole, where $u$ is the unit vector pointing towards $v_n$.  This is
  because an all positive vector such as $u$ and any nonzero
  nonnegative vector have a positive dot product, making the angle
  between them acute.  Finally, hemispheric contraction to $u$
  monotonically contracts $\lk(v,O)$ because $\lk(v,O)$ is a convex
  set containing $u$.
\end{proof}

The orthogonality embedded in the definition of an orthoscheme causes
its face links to decompose into spherical joins.  As a warm-up for
the general statement, we consider links of vertices and edges in
orthoschemes.

\begin{exmp}[Vertex links in orthoschemes]
  Let $O=\ortho(v_0,v_1,\ldots,v_n)$ be an orthoscheme, let $v=v_k$ be
  a vertex with $0<k<n$, and consider suborthoschemes $K=O(0,k)$ and
  $L=O(k,n)$.  The affine subspaces containing $K$ and $L$ are
  orthogonal to each other and the original orthoscheme is the convex
  hull of these two faces.  Thus $\lk(v,O) = \lk(v,K) \ast \lk(v,L)$.
  In fact, this formula remains valid when $k=0$ (or $k=n$) since then
  $K$ (or $L$) has only a single point, its link is empty and this
  factor drops out of the spherical join since $\emptyset$ is the
  identity of the spherical join operation.  Finally, note that the
  factors are endpoint links of the suborthoschemes $K$ and $L$.
\end{exmp}

\begin{exmp}[Edge links in orthoschemes]
  For each $k\leq \ell$ consider the link of the edge $e_{k\ell}$
  connecting $v_k$ and $v_\ell$ in $O=\ortho(v_0,v_1,\ldots,v_n)$.  If
  we define $K = O(0,k)$, $L = O(k,\ell)$ and $M= O(\ell,n)$ then we
  claim that $\lk(e_{k\ell},O) = \lk(v_k,K) \ast \lk(e_{k\ell},L) \ast
  \lk(v_\ell,M)$.  To see this note that the linear subspaces
  corresponding to the affine spans of $K$, $L$ and $M$ form an
  orthogonal decomposition of $\R^n$ and, as a consequence, any vector
  can be uniquely decomposed into three orthogonal components.  It is
  then straightforward to see that a vector perpendicular to
  $e_{k\ell}$ points into $O$ iff its components live in the specified
  links.  The first and last factors are local endpoint links and the
  middle factor is a local diagonal link.  As above, the first factors
  drops out when $k=0$, the last factor drops out when $\ell=n$, but
  also note that the middle factor drops out when $k$ and $\ell$ are
  consecutive since the diagonal link of $O(k,k+1)$ is empty.
\end{exmp}

We are now ready for the precise general statement.

\begin{prop}[Links in orthoschemes]\label{prop:os-links}
  Let $O=\ortho(v_0,v_1,\ldots,v_n)$ and let $\sigma$ be a
  $k$-dimensional face with vertices $\{v_{x_0}, v_{x_1}, \ldots,
  v_{x_k}\}$ where $0\leq x_0 < x_1 < \cdots <x_k \leq n$.  The link
  of $\sigma$ in $O$ is a spherical join of two endpoint links and $k$
  diagonal links of suborthoschemes.  More specifically, \[\lk(\sigma,
  O) = K_\zero \ast L_0 \ast L_1 \ast \cdots \ast L_k \ast K_\one\]
  where $K_\zero = \lk(v_{x_0}, O(\zero, x_0))$ and $K_\one =
  \lk(v_{x_k}, O(x_k, \one))$ are local endpoint links, and each $L_i
  = \lk(e_{x_{i-1}x_i}, O(x_{i-1}, x_i))$ is a local diagonal link.
\end{prop}

\begin{proof}
  The full proof is basically the same as the one given above for edge
  links.  Orthogonally decompose $\R^n$ into linear subspaces
  corresponding to the affine hulls of $O(\zero,x_0)$, $O(x_k,\one)$
  and $O(x_{i-1},x_i)$ for each $i\in [k]$, then check that a vector
  perpendicular to $\sigma$ points into $O$ iff its components live in
  the listed links.  In an orthogonal basis compatible with the
  orthogonal decomposition that contains the local diagonal directions
  along each $e_{x_{i-1}x_i}$ this conclusion is immediate.
\end{proof}

Coxeter's interest in these shapes is related to the observation that
the barycentric subdivision of a regular polytope decomposes it into
isometric orthoschemes.  These orthoschemes are fundamental domains
for the action of its isometry group and correspond to the chambers of
its Coxeter complex.  For example, a barycentric subdivision of the
$3$-cube of side length~$2$ partitions it into $48$ standard
$3$-orthoschemes and the barycentric subdivision of an $n$-cube of
side length~$2$ produces $n!\cdot 2^n$ standard $n$-orthoschemes (see
Figure~\ref{fig:cube}).  These cube decompositions also make it easy
to identify the shape of the endpoint link and the diagonal link in a
standard orthoscheme.

\begin{figure}
\includegraphics[width=3in]{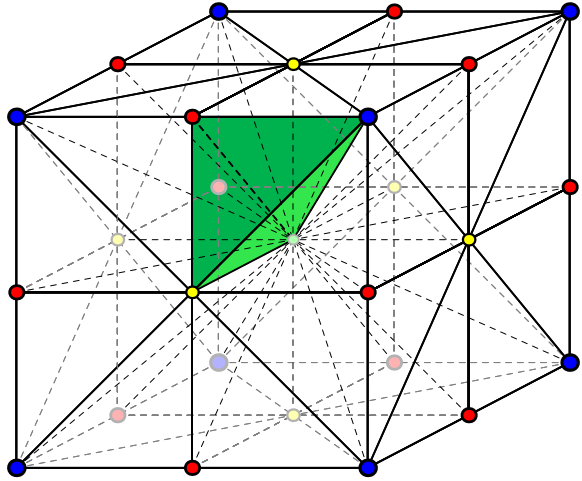}
\caption{Barycentric subdivision of a $3$-cube of side length~$2$ into
  $48$ standard $3$-orthoschemes.\label{fig:cube}}
\end{figure}

\begin{defn}[Coxeter shapes]\label{defn:special}
  Let $O=\ortho(v_0,v_1,\ldots,v_n)$ be a standard $n$-orthoscheme.
  The links $\lk(v_0,O)$ and $\lk(v_n,O)$, are isometric to each other
  and we call this common shape $\beta_n$ because it is a spherical
  simplex known as the \emph{Coxeter shape of type $B_n$}.  The type
  $B_n$ Coxeter group is the isometry group of the $n$-cube and the
  barycentric subdivision of the $n$-cube mentioned above shows that
  $\beta_n$ is its Coxeter shape, i.e. a fundamental domain for the
  action of this group on the sphere.  In low dimensions, $\beta_0 =
  \emptyset$, $\beta_1$ is a point, $\beta_2$ is an arc of length
  $\frac{\pi}{4}$ and $\beta_3$ is a spherical triangle with angles
  $\frac{\pi}{2}$, $\frac{\pi}{3}$, and $\frac{\pi}{4}$.  Similarly,
  the link of the diagonal connecting $v_0$ and $v_n$ in $O$ is a
  shape that we call $\alpha_{n-1}$ because it is the spherical
  simplex known as the \emph{Coxeter shape of type $A_{n-1}$}.  The
  $A_{n-1}$ Coxeter group is the symmetric group $\sym_n$, the
  isometry group of the regular $(n-1)$-simplex and also the
  stabilizer of a vertex inside the isometry group of the $n$-cube.
  In low dimensions, $\alpha_0 = \emptyset$, $\alpha_1$ is a point,
  $\alpha_2$ is an arc of length $\frac{\pi}{3}$ and $\alpha_3$ is a
  spherical triangle with angles $\frac{\pi}{2}$, $\frac{\pi}{3}$, and
  $\frac{\pi}{3}$.
\end{defn}

The following corollary of Proposition~\ref{prop:os-links} is now immediate.

\begin{cor}[Links in standard orthoschemes]\label{cor:ortho-links}
  The link of a face of a standard orthoscheme is a spherical join of
  spherical simplices each of type $A$ or type $B$.  More
  specifically, in the notation of Proposition~\ref{prop:os-links},
  $K_\zero$ has shape $\beta_{x_0}$, $K_\one$ has shape
  $\beta_{n-x_n}$ and $L_i$ has shape $\alpha_j$ with $j=x_{i+1}-x_i -
  1$.
\end{cor}

As an illustration, the link of the tetrhedron with corners $v_3$,
$v_6$, $v_7$ and $v_{12}$ in a standard $12$-orthoscheme is isometric
to $\beta_3 \ast \alpha_2 \ast \alpha_0 \ast \alpha_4 \ast \beta_0 =
\beta_3 \ast \alpha_2 \ast \alpha_4$.  Finally we record a few results
about lengths and angles in orthoscheme links that are needed later in
the article.

\begin{prop}[Edge lengths]\label{prop:edge-length}
  The link of the diagonal of a standard $n$-orthoscheme $O =
  \ortho(v_0,\ldots,v_n)$ is a spherical simplex of shape
  $\alpha_{n-1}$ whose vertices correspond to the $v_i$ with $0<i<n$.
  Moreover, if $i$, $j$, and $k$ are positive integers with $i+j+k=n$
  and $\theta$ is the length of the edge connecting the vertex of rank
  $i$ and the vertex of corank $k$ in $\alpha_{n-1}$, then $0 < \theta
  < \frac{\pi}{2}$ and $\cos(\theta) = \sqrt{ \frac{i}{i+j}\cdot
    \frac{k}{j+k}}$.
\end{prop}

\begin{proof}
  Consider the triangle with vertices $v_0$, $v_i$ and $v_{n-k} =
  v_{i+j}$.  If we project this triangle onto the hyperplane
  perpendicular to the edge $e_{0n}$, then the angle at $v_0$ in the
  projected triangle is the length of the corresponding edge in
  $\alpha_{n-1}$.  Let $u$ be the projection of $e_{0i}$ and let $v$
  be the projection of $e_{0(n-k)}$ (with both multiplied by $n$ to
  clear the denominators).  In coordinates $u = ((j+k)^i,(-i)^{j+k})$
  and $v = (k^{i+j},(-i-j)^k)$.  Here we are using Conway's exponent
  notation to simplify vector expressions.  In words the first $i$
  coordinates of $u$ are $j+k$ and the remaining $j+k$ coordinates are
  $-i$.  The dot products simplify as follows: $u \cdot v = i.k.n$
  while $u \cdot u = i.(j+k).n$ and $v \cdot v = (i+j).k.n$.  Thus
  \[ \cos^2(\theta) = \frac{(u\cdot v)(u\cdot v)}{(u\cdot u)(v\cdot
    v)} = \frac{(i.k.n)(i.k.n)}{(i.(j+k).n)((i+j).k.n)} =
  \frac{i}{i+j}\cdot\frac{k}{j+k}. \]
\end{proof}

\begin{cor}[Spherical triangles]\label{cor:triangles}
  Let $\ortho(v_0,\ldots,v_n)$ be a standard $n$-orthoscheme.  For
  each $0 < x < y < z < n$, the diagonal link of the suborthoscheme
  $\ortho(v_0,v_x,v_y,v_z,v_n)$ is a spherical triangle with acute
  angles at $v_x$ and $v_z$ and a right angle at $v_y$.
\end{cor}

\begin{proof}
  From Proposition~\ref{prop:edge-length} the lengths of the edges of the
  spherical triangle are known explicitly and the angle assertions
  follow from the standard spherical trigonometry.  For example, if we
  select positive integers $i+j+k+l=n$ such that $v_x$ has rank $i$,
  $v_y$ has rank $i+j$ and $v_z$ has rank $i+j+k$, then
  $\cos\order{e_{xy}} = \sqrt{\frac{i}{i+j} \cdot \frac{k+l}{j+k+l}}$,
  $\cos\order{e_{yz}} = \sqrt{\frac{i+j}{i+j+k} \cdot \frac{l}{k+l}}$
  and $\cos\order{e_{xz}} = \sqrt{\frac{i}{i+j+k} \cdot
    \frac{l}{j+k+l}}$.  From the equality $\cos\order{e_{xy}}\cdot
  \cos\order{e_{yz}} = \cos\order{e_{xz}}$ and the spherical law of
  cosines we infer that the angle at $v_y$ is a right angle.  The
  acute angle conclusion involves a similar but messier computation.
\end{proof}

%%%%%%%%%%%%%%%%%%%%%%%%%%%%%%%%%%%%%%%
\section{Orthoscheme complexes}
%%%%%%%%%%%%%%%%%%%%%%%%%%%%%%%%%%%%%%%

In this section we use orthoschemes to turn order complexes into \pe
complexes. Although every simplicial complex can be turned into a \pe
complex by making simplices regular and every edge length~$1$, the
curvature properties of the result tend to be hard to evaluate.  For
order complexes of graded posets orthoschemes are a better option.

\begin{defn}[Orthoscheme complex]
  The \emph{orthoscheme complex} of a graded poset $P$ is a metric
  version of its order complex $\order{P}$ that assigns every top
  dimensional simplex in $\order{P}$ (i.e. those corresponding to
  maximal chains $x_0 < x_1 < \cdots < x_n$) the metric of a standard
  orthoscheme with $v_{x_i}$ corresponding to $v_i$.  As a result, for
  all elements $x < y$ in $P$, the length of the edge connecting $v_x$
  and $v_y$ in $\order{P}$ is $\sqrt{k}$ where $k$ is the rank of
  $P_{xy}$.  When $\order{P}$ is turned into a \pe complex in this way
  we say that $\order{P}$ is endowed with the \emph{orthoscheme
    metric}.  Unless otherwise specified, from now on $\order{P}$
  indicates an orthoscheme complex, i.e. an order complex with the
  orthoscheme metric.
\end{defn}

One reason for using this particular metric on the order
complex of a poset is that it turns standard examples of posets into
metrically interesting \pe complexes.

 \begin{figure}
 \begin{tabular}{ccc}
 \bt{c}\includegraphics{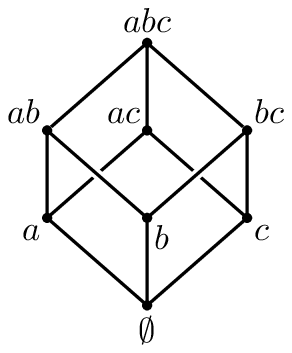}\et & \hspace{.5in}
 &\bt{c}\includegraphics{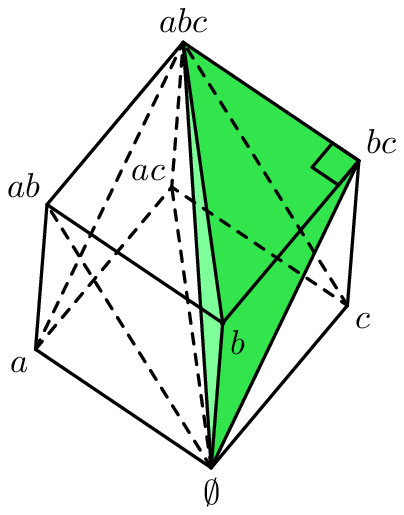}\et 
 \end{tabular}
 \caption{The rank~$3$ boolean lattice and its unit $3$-cube
   orthoscheme complex.  The orthoscheme from the chain $\emptyset
   \subset \{b\} \subset \{b,c\} \subset \{a,b,c\}$ is
   shaded.\label{fig:boolean}}
 \end{figure}

\begin{exmp}[Boolean lattices]
  The \emph{rank $n$ boolean lattice} is the poset of all subsets of
  $[n] := \{1,2,\ldots,n\}$ ordered by inclusion.  The orthoscheme
  complex of a boolean lattice is a subdivided unit $n$-cube (or one
  orthant of the barycentric subdivision of the $n$-cube of side
  length~$2$ described earlier), its endpoint link is the barycentric
  subdivision of an all-right spherical simplex tiled by simplices of
  shape $\beta_n$ and its diagonal link is a subdivided sphere tiled
  by simplices of shape $\alpha_{n-1}$.  See Figure~\ref{fig:boolean}.
\end{exmp}

That the orthoscheme complex of a boolean lattice is a cube can be
explained by a more general fact about products.

\begin{rem}[Products]
  A product of posets produces an orthoscheme complex that is a
  product of metric spaces.  In particular, if $P$ and $Q$ are bounded
  graded posets, then $\order{P \times Q}$ and $\order{P} \times
  \order{Q}$ are isometric.  The product on the left is a product of
  posets and the product on the right is a product of metric spaces.
  Since finite boolean lattices are poset products of two element
  chains, their orthoscheme complexes are, up to isometry, products of
  unit intervals, i.e. cubes.
\end{rem}

Cube complexes produce a second family of examples.

\begin{exmp}[Cube complexes]
  Let $K$ be a cube complex scaled so that every edge has length $2$.
  The face poset of $K$ is a graded poset $P$ whose intervals are
  boolean lattices.  The orthoscheme complex $\order{P}$ is isometric
  to the cube complex $K$.  In other words, the metric barycentric
  subdivision of an arbitrary cube complex is identical to the
  orthoscheme complex of its face poset.
\end{exmp}

A third family of examples shows that there are interesting $\cat(0)$
orthoscheme complexes unrelated to cubes and cube complexes.

\begin{exmp}[Linear subspace posets]\label{exmp:linear}
  The $n$-dimensional \emph{linear subspace poset} over a field $\F$
  is the poset $L_n(\F)$ of all linear subspaces of the
  $n$-dimensional vector space $\F^n$ ordered by inclusion.  It's
  basic properties are explored in Chapter~$3$ of \cite{EC1}. The
  poset $L_n(\F)$ is bounded above by $\F^n$ and below by the trivial
  subspace and it is graded by the dimension of the subspace an
  element of $L_n(\F)$ represents.  It turns out that the orthoscheme
  complex of $L_n(\F)$ is a $\cat(0)$ space and its diagonal link is a
  standard example of a thick spherical building of type $A_{n-1}$.
  The smallest nontrivial example, with $\F = \Z_2$ and $n=3$, is
  illustrated in Figure~\ref{fig:linear} along with its diagonal link.
  The middle levels of $L_3(\Z_2)$ correspond to the $7$ points and
  $7$ lines of the projective plane of order $2$ and its diagonal link
  is better known as the Heawood graph, or as the incidence graph of
  the Fano plane.
\end{exmp}

\begin{figure}
\begin{tabular}{cc}
\bt{c}\includegraphics{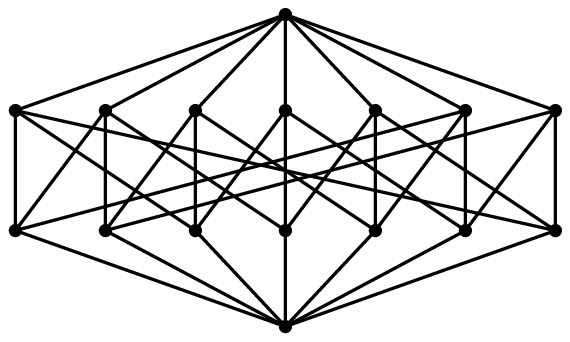}\et
&\bt{c}\includegraphics{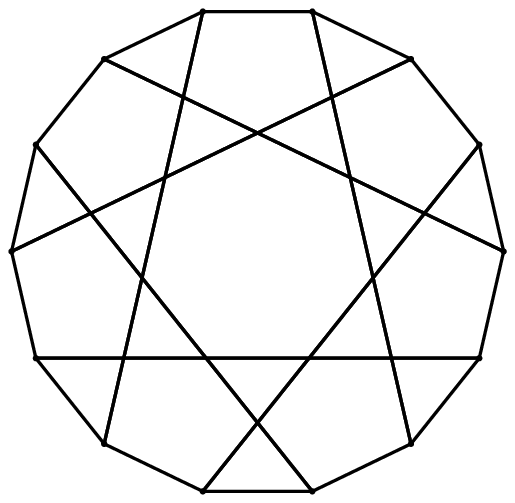}\et
\end{tabular}
\caption{The poset $L_3(\Z_2)$ and its diagonal link.  In the \ps
  complex on the right, every edge in the graph should be viewed as a
  spherical arc of length $\frac{\pi}{3}$.\label{fig:linear}}
\end{figure}

With these examples in mind, we now turn to the question of when the
orthoscheme complex of a bounded graded poset $P$ is a $\cat(0)$ \pe
complex.  The first step is to examine some of its more elementary
links.

\begin{defn}[Elementary links]
  Let $P$ be a bounded graded poset of rank~$n$.  Three links in the
  orthoscheme complex $\order{P}$ are of particular interest.  The \ps
  complexes $\lk(v_\zero,\order{P})$ and $\lk(v_\one,\order{P})$ are
  the \emph{endpoint links} of $\order{P}$ and $\lk(e_{\zero\one},
  \order{P})$ is its \emph{diagonal link}.  The endpoint links are \ps
  complex built out of copies of $\beta_n$ and the diagonal link is a
  \ps complex built out of copies of $\alpha_{n-1}$.  In fact,
  $\lk(v_\zero,\order{P})$ is the simplicial complex
  $\order{P\setminus \{\zero\}}$ with a \ps $\beta_n$ metric applied
  to each maximal simplex.  Similarly, $\lk(e_{\zero\one},\order{P})$
  is $\order{P\setminus \{\zero,\one\}}$ with an $\alpha_{n-1}$ metric
  on each maximal simplex.  Collectively these three links are the
  \emph{elementary links} of the orthoscheme complex $\order{P}$.
  Note that endpoint links are empty when $P$ has rank~$0$, and the
  diagonal links are empty when $P$ has rank~$1$.  This corresponds to
  the fact that $\beta_0 = \alpha_0 = \emptyset$.
\end{defn}

In order to determine whether an orthoscheme complex is $\cat(0)$, we
need to understand the structure of the link of an arbitrary simplex.
We do this by showing that the link of an arbitrary simplex decomposes
as a spherical join of local elementary links.  This decomposition is
based on the decomposition described in Corollary~\ref{cor:ortho-links} and
is only possible because of the orthogonality built into the
definition of an orthoscheme.

\begin{prop}[Links in orthoscheme complexes]\label{prop:osc-links}
  Let $P$ be a bounded graded poset.  If $x_0 < x_1 < \cdots < x_k$ is
  a chain in $P$ and $\sigma$ is the corresponding simplex in its
  orthoscheme complex, then $\lk(\sigma, \order{P})$ is a spherical
  join of two local endpoint links and $k$ local diagonal links.  More
  specifically, \[\lk(\sigma, \order{P}) = K_\zero \ast L_0 \ast L_1
  \ast \cdots \ast L_k \ast K_\one\] where $K_\zero = \lk(v_{x_0},
  \order{P(\zero, x_0)})$ and $K_\one = \lk(v_{x_k}, \order{P(x_k,
    \one)})$ are local endpoint links, and each $L_i =
  \lk(e_{x_{i-1}x_i}, \order{P(x_{i-1},x_i)})$ is a local diagonal
  link.
\end{prop}

\begin{proof}
  Let $P$ have rank $n$.  Since every chain is contained in a maximal
  chain of length $n$, every simplex containing $\sigma$ is contained
  in a standard $n$-orthoscheme of $\order{P}$.  In particular, the
  link of $\sigma$ in $\order{P}$ is a \ps complex obtained by gluing
  together the link of $\sigma$ in each $n$-orthoscheme that contains
  it.  By Corollary~\ref{cor:ortho-links}, each such link decomposes as
  spherical joins of Coxeter shapes.  Moreover, these decompositions
  are compatible from one $n$-orthoscheme to the next, reflecting the
  fact that every maximal chain extending $x_0 < x_1 < \cdots < x_k$
  is formed by selecting a maximal chain from $P(\zero,x_0)$, a
  maximal chain from $P(x_k,\one)$ and a maximal chain from
  $P(x_{i-1},x_i)$ for each $i \in [k]$ and these choices can be made
  independently of one another.  When the links of $\sigma$ in each
  $n$-orthoscheme are pieced together, the result is the one listed in
  the statement.
\end{proof}

Understanding links thus reduces to understanding elementary links.

\begin{lem}[Endpoint links]\label{lem:endpt-links}
  An endpoint link of a bounded graded poset $P$ is a monotonically
  contractible \ps complex and thus contains no unshrinkable loops. In
  particular, endpoint links cannot be not quite $\cat(1)$.
\end{lem}

\begin{proof}
  Let $v$ be an endpoint of $\order{P}$, let $K = \lk(v,\order{P})$,
  then let $u \in K$ be the unit vector at $v$ pointing along the
  common diagonal of all the orthoschemes of $\order{P}$.  The
  contractions defined in Proposition~\ref{prop:endpt-contract} are
  compatible on their overlaps and jointly define a monotonic
  contraction from $K$ to $u$.  In particular, all loops in $K$
  monotonically shrink under this homotopy.
\end{proof}

\begin{lem}[Diagonal links]\label{lem:lk-interval}
  Let $P$ be a bounded graded poset.  For every $x<y$ in $P$ there is
  a simplex $\sigma$ in $\order{P}$ such that $\lk(\sigma,\order{P})$
  and $\lk(e_{xy},\order{P_{xy}})$ are isometric.
\end{lem}

\begin{proof}
  Pick a maximal chain extending $x<y$ and remove the elements
  strictly between $x$ and $y$.  If $\sigma$ is the simplex of
  $\order{P}$ that corresponds to this subchain then by
  Proposition~\ref{prop:osc-links} the link of $\sigma$ is a spherical join
  of $\lk(e_{xy},\order{P_{xy}})$ with other elementary links, all of
  which are empty in this context.  As a consequence
  $\lk(\sigma,\order{P})$ and $\lk(e_{xy},\order{P_{xy}})$ are
  isometric.
\end{proof}

Recall that a \ps complex is large if it has no short local geodesic
loops.  Using the results above, we now show that the curvature of an
orthoscheme complex only depends on whether or not its local diagonal
links are large.

\begin{thm}[Orthoscheme link condition]\label{thm:ortho-curv}
  If $P$ is a bounded graded poset then its orthoscheme complex
  $\order{P}$ is not $\cat(0)$ iff there is a local diagonal link of
  $P$ that is not quite $\cat(1)$.  As a result, $\order{P}$ is
  $\cat(0)$ iff every local diagonal link of $P$ is large.
\end{thm}

\begin{proof}
  For each interval $P_{xy}$ there is a simplex $\sigma$ in
  $\order{P}$ so that $\lk(\sigma,\order{P})$ and $\lk(e_{xy},
  \order{P_{xy}})$ are isometric (Lemma~\ref{lem:lk-interval}).  If
  $\order{P}$ is $\cat(0)$, then $\lk(\sigma, \order{P})= \lk(e_{xy},
  \order{P_{xy}})$ is $\cat(1)$.  Conversely, suppose the complex
  $\order{P}$ is not $\cat(0)$ and recall that it is contractible
  (Proposition~\ref{prop:contractible}).  It contains a simplex with a not
  quite $\cat(1)$ link (Proposition~\ref{prop:curv-lks}) that factors as a
  spherical join of local elementary links
  (Proposition~\ref{prop:osc-links}).  Moreover, there is a not quite
  $\cat(1)$ factor (Proposition~\ref{prop:cat1-joins}) which must be a
  diagonal link of an interval since by Lemma~\ref{lem:endpt-links}
  endpoint links cannot be not quite $\cat(1)$.
\end{proof}

%%%%%%%%%%%%%%%%%%%%%%%%%%%%%%%%%%%%%%%%
\section{Spindles}\label{sec:spindles}
%%%%%%%%%%%%%%%%%%%%%%%%%%%%%%%%%%%%%%%%

In this section we introduce combinatorial configurations we call
spindles and we relate the existence of spindles in a bounded graded
poset $P$ to the existence of certain local geodesic loops in a local
diagonal link of $P$.  When defining spindles, we use the notion of
complementary elements.

\begin{defn}[Complements]
  Two elements $x$ and $y$ in a bounded poset $P$ are
  \emph{complements} or \emph{complementary} when $x \vee y = \one$
  and $x \wedge y = \zero$.  In particular $\one$ is their only upper
  bound and $\zero$ is their only lower bound.  A pair of elements in
  a boolean lattice representing complementary subsets are
  complementary in this sense.  Note that if $z$ is any maximal lower
  bound of $x$ and $y$ and $w$ is any minimal upper bound of $x$ and
  $y$ then $x$ and $y$ are complements in the interval $P_{zw}$.
\end{defn}

\begin{figure}
  \bt{cc}
  \bt{c}\includegraphics{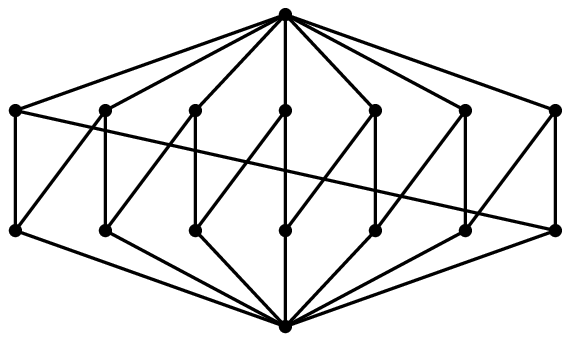}\et &
  \bt{c}\includegraphics{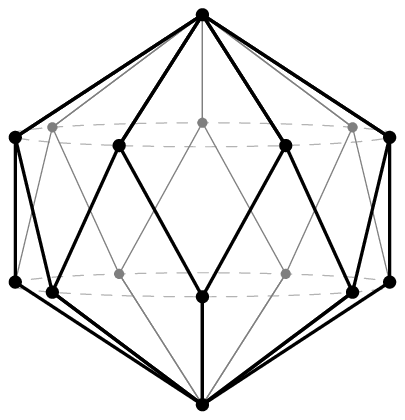}\et  
  \et
  \caption{Two views of a spindle of girth $14$.  The $3$-dimensional
    version, which looks like an antiprism capped off by two pyramids
    is the reason for the name.\label{fig:spindle}}
\end{figure}

\begin{defn}[Spindles]
  For some $k \geq 2$ let $(x_1,x_2,\ldots,x_{2k})$ be a sequence of
  $2k$ distinct elements in a bounded graded poset $P$ where the
  subscripts are viewed mod $2k$ and note that the parity of a
  subscript is well-defined in this context.  Such a sequence forms a
  \emph{global spindle of girth $2k$} if for every $i$ with one
  parity, $x_{i-1}$ and $x_{i+1}$ are complements in $P_{\zero x_i}$
  and for every $i$ with the other parity, $x_{i-1}$ and $x_{i+1}$ are
  complements in $P_{x_i \one}$.  See Figure~\ref{fig:spindle}. The
  elements $\zero$ and $\one$ are called the \emph{endpoints} of the
  global spindle and the sequence of elements
  $(x_1,x_2,\ldots,x_{2k})$ describes a zig-zag path in $P$ but with
  additional restrictions.  There is also a local version.  A (local)
  \emph{spindle}, with or without the adjective, is a global spindle
  inside some interval $P_{zw}$ with endpoints $z$ and $w$.
\end{defn}

\begin{defn}[Short spindles]\label{def:short-spindles}
  The \emph{length} of a global spindle is the length of the
  corresponding loop in the $1$-skeleton of the diagonal link of $P$
  (endowed with the metric induced by the orthoscheme metric). A
  global spindle is \emph{short} if its length is less than $2\pi$.
  The lengths of the individual edges in the diagonal link can be
  calculated using Proposition~\ref{prop:edge-length} and the reader should
  note that every edge in a diagonal link has length less than
  $\frac{\pi}{2}$.  Thus every spindle of girth~$4$ is short.
\end{defn}

A spindle is a generalization of a bowtie in the following sense.

\begin{prop}[Spindles, bowties and lattices]\label{prop:spindle-bowtie}
  A bounded graded poset $P$ contains a bowtie iff it contains a
  spindle of girth~$4$.  In particular, $P$ is a lattice iff $P$ does
  not contains a spindle of girth~$4$ and every bounded graded poset
  with no short spindles is a lattice.
\end{prop}

\begin{proof}
  If $P$ contains a spindle of girth~$4$ in the interval $P_{zw}$,
  then $x_1$, $x_2$, $x_3$ and $x_4$ form a bowtie since the bowtie
  conditions follow from the complementarity requirements.  For
  example, $x_2$ is a maximal lower bound for $x_1$ and $x_3$ because
  $x_1$ and $x_3$ are complements in the interval $P_{x_2 w}$.  In the
  other direction suppose $a$, $b$, $c$, and $d$ form a bowtie, let
  $z$ be any maximal lower bound for $b$ and $d$ and let $w$ be any
  minimal upper bound for $a$ and $c$.  It is then easy to check that
  $(a,b,c,d)$ is a global spindle of girth~$4$ in the interval
  $P_{zw}$.  The final assertion follows from
  Proposition~\ref{prop:lattice-bowties} and the fact that every spindle of
  girth~$4$ is short.
\end{proof}

The next step is to relate spindles and local geodesics in diagonal
links.  Let $P$ be a bounded graded poset and let $K$ be its
orthoscheme complex.  By Theorem~\ref{thm:ortho-curv} we know that
determining whether or not $K$ is $\cat(0)$ reduces to determining
whether or not an interval of $P$ has a diagonal link containing a
short local geodesic.  Since generic local geodesics are hard to
describe and hard to detect, we shift our focus to the simplest local
geodesics, i.e. those that remain in the $1$-skeleton of the local
diagonal link.  We now show that such paths are described by spindles.
Figure~\ref{fig:loops} summarizes the relationships among these three
classes of loops just described.

\begin{figure}
\begin{tabular}{ccccc}
  $\left\{ \begin{array}{c} \textrm{Local geodesics}\\ \textrm{in a
        local}\\ \textrm{diagonal link} \end{array}\right\}$
  & $\supset$ &
  $\left\{ \begin{array}{c} \textrm{Local geodesics}\\ \textrm{that
        remain in}\\ \textrm{its $1$-skeleton} \end{array}\right\}$
  & $\subset$ &
  $\left\{ \begin{array}{c} \textrm{Loops that}\\ \textrm{correspond}\\
      \textrm{to spindles} \end{array}\right\}$
\end{tabular}
\caption{The three types of loops discussed in this
  section.\label{fig:loops}}
\end{figure}

\begin{prop}[Transitions and complements]\label{prop:transitions}
  Let $P$ be a bounded graded poset and let $e_{xy}$ and $e_{yz}$ be
  distinct edges in the diagonal link of $P$. If the piecewise
  geodesic path from $v_x$ to $v_y$ to $v_z$ in the diagonal link of
  $P$ is locally geodesic at $v_y$ then either $x$ and $z$ are
  complements in the interval $P_{\zero y}$ or $x$ and $z$ are
  complements in the interval $P_{y \one}$.  As a consequence, every
  local geodesic loop in the diagonal link of $P$ that remains in its
  $1$-skeleton corresponds to a global spindle of $P$.
\end{prop}

\begin{proof}
  First note that because the edges $e_{xy}$ and $e_{yz}$ exist, $x$
  and $y$ are comparable in $P$ and $y$ and $z$ are comparable in $P$.
  If $x$ and $z$ are comparable as well then the path through $v_y$ is
  not locally geodesic because $x$, $y$ and $z$ form a chain, $v_x$,
  $v_y$ and $v_z$ are the corners of a convex spherical triangle in
  $\lk(e_{\zero \one},\order{P})$ and the nonobtuse angle at $v_y$
  (Corollary~\ref{cor:triangles}) shows that the path through $v_y$ is
  not locally geodesic because the transition vectors are not far
  apart.  In the remaining cases both $x$ and $z$ are below $y$ or
  both $x$ and $z$ are above $y$.  Assume that both are in $P_{\zero
    y}$; the other case is analogous and omitted.  If there is a $w$
  in $P_{\zero y}$ that is an upper bound of $x$ and $z$ other than
  $y$ or a lower bound of $x$ and $z$ other than $\zero$ then there is
  a spherical triangle in $\lk(e_{\zero \one},\order{P})$ with
  vertices $v_w$, $v_y$ and $v_x$ and a second triangle with vertices
  $v_w$, $v_y$ and $v_z$.  As both triangles have acute angles at
  $v_y$ (Corollary~\ref{cor:triangles}) and the path through $v_y$ is
  not locally geodesic because the transition vectors are not far
  apart.  For the final assertion suppose that $(x_1,\ldots,x_k)$ are
  the vertices of a local geodesic loop that remains in the
  $1$-skeleton of the diagonal link of $P$.  The local result proved
  above means that adjacent triples satisfy the required conditions
  and it forces the orderings ($x_i < x_{i+1}$ or $x_i > x_{i+1}$) to
  alternate, making $k$ even.
\end{proof}

It is important to note that implication established above is in one
direction only: a locally geodesic loop in the $1$-skeleton of a
diagonal link must come from a spindle but not every spindle
necessarily produces a locally geodesic loop.  The problem is that
just because $x$ and $z$ are complements in $P_{\zero y}$ does not
necessarily mean that $v_x$ and $v_z$ are far apart in
$\lk(v_y,\order{P})$ even though we conjecture that this is often the
case.

\begin{conj}[Complements are far apart]\label{conj:complements}
  Let $P$ be a bounded graded poset and let $K$ be its diagonal link.
  If $x$ and $y$ are complements in $P$ and $K$ is $\cat(1)$ then
  $v_x$ and $v_y$ are far apart in $K$.
\end{conj}

We know that Conjecture~\ref{conj:complements} is true for the
rank~$n$ boolean lattice $P$ because the only elements that are
complements in $P$ correspond to complementary subsets $A$ and $B$,
these correspond to opposite corners of the $n$-cube $\order{P}$ and
to antipodal points in the $n-1$ sphere that is the diagonal link of
$P$.  In particular, they represent points that are distance $\pi$
from each other in $\lk(e_{\zero \one},\order{P})$.  In fact, for
boolean lattices, more is true.

\begin{prop}[Boolean spindles]
  If $P$ is a boolean lattice of rank~$n$ then every spindle in $P$
  has girth~$6$, length~$2\pi$ and describes an equator of the
  $(n-1)$-sphere that is the diagonal link of $P$.  In particular, $P$
  has no short spindles.
\end{prop}

\begin{proof}
  Let $(x_1,x_2,\ldots, x_{2k})$ be a spindle in $P$.  Since intervals
  in boolean lattices are themselves boolean lattices, we may assume
  without loss of generality that this is a global spindle.  Suppose
  $x_1 < x_2$ and let $A$ and $B$ be the uniquely determined disjoint
  subsets of $[n]$ such that $x_1$ represents the set $A$ and $x_2$
  represents the subset $A \cup B$.  Finally let $C$ be the complement
  of $A \cup B$ in $[n]$.  Since $x_3$ is a complement of $x_1$ in
  $P_{\zero x_2}$ and complements in boolean lattices are unique,
  $x_3$ corresponds to the set $B$.  Similarly, $x_4$ is a complement
  of $x_2$ in $P_{x_3 \one}$ and thus must correspond to the set $B
  \cup C$.  Continuing in this way, $x_5$, $x_6$, $x_7$ and $x_8$
  correspond to the sets $C$, $A \cup C$, $A$ and $A \cup B$
  respectively.  Since the elements in a spindle are distinct, $x_i =
  x_{i+6}$ for all $i$ and the spindle has girth~$6$.  To see that its
  length $2\pi$, note that the fact that complements are far apart in
  boolean lattices means that global spindles describe paths that are
  embedded local geodesics in the diagonal link.  In this case the
  diagonal link of $P$ is a sphere and the only embedded local
  geodesics are equatorial paths of length $2\pi$.
\end{proof}

We conclude this section by extending this result to modular lattices.

\begin{defn}[Modular lattices]
  A \emph{modular lattice} is a graded lattice with the property that
  if $x$ and $y$ are complements in an interval $P_{zw}$ and $x$ has
  rank $i$ and corank $j$ in this interval, then $y$ has corank $i$
  and rank $j$ in this interval.  It should be clear from this
  definition that finite rank boolean lattices are examples of modular
  lattices as are the linear subspace posets described in
  Example~\ref{exmp:linear}.
\end{defn}

\begin{figure}
  \includegraphics{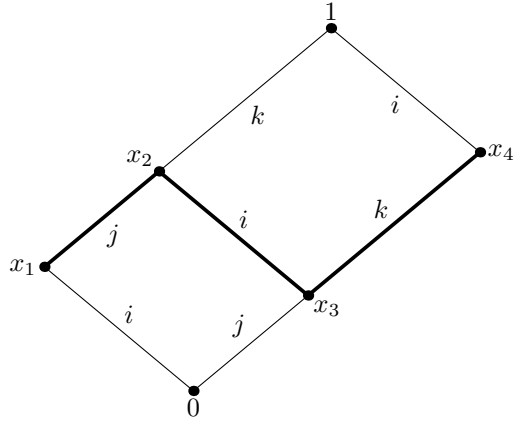}
  \caption{A portion of a global spindle in a modular lattice that
    describes a local geodesic edge path of length $\pi$ in its diagonal
    link.\label{fig:pi-path}}
\end{figure}

\begin{prop}[Modular spindles]\label{prop:modular}
  If $P$ is a bounded graded modular lattice then every spindle in $P$
  has girth at least $6$ and describes a loop of length at least
  $2\pi$.  In particular, $P$ has no short spindles.
\end{prop}

\begin{proof}
  Since $P$ is a lattice, by Proposition~\ref{prop:spindle-bowtie} there are
  no spindles of length~$4$.  Thus every spindle has girth at
  least~$6$. Next, since intervals in modular lattices are modular
  lattices we only need to consider global spindles.  Let
  $(x_1,x_2,\ldots, x_{2k})$ be a global spindle with $x_1 < x_2$ and
  let $i$, $j$ and $k$ be positive integers such that $x_1$ has rank
  $i$, $x_2$ has corank $k$ and $i+j+k=n$ where $n$ is the rank of
  $P$.  The complementarity conditions and the definition of
  modularity imply that $x_3$ has rank $j$ and $x_4$ has corank $i$.
  See Figure~\ref{fig:pi-path}.  The key observation is that these are
  the same ranks and coranks and one of the geodesic paths between
  complementary subsets in a boolean lattice.  In particular, the sum
  of the lengths of these edges in the diagonal link of $P$ is exactly
  $\pi$.  Since the girth of the spindle is at least~$6$, its length
  is at least $2\pi$.  This completes the proof.
\end{proof}

Since bounded graded modular lattices have no short spindles, the
poset curvature conjecture leads us to conjecture the following.

\begin{conj}[Modular lattices and $\cat(0)$]
  Every bounded graded modular lattice has a $\cat(0)$ orthoscheme
  complex.
\end{conj}

%%%%%%%%%%%%%%%%%%%%%%%%%%%%%%%%%%%
\section{Low rank}\label{sec:low}
%%%%%%%%%%%%%%%%%%%%%%%%%%%%%%%%%%%

In this section we shift our attention from bounded posets of
arbitrary rank to those of rank at most~$4$.  Our goal is to prove the
poset curvature conjecture in this context, thus establishing
Theorem~\ref{main:poset}.  The proof depends on two basic results.
The first is that complementary elements in a low rank poset
correspond to vertices that are always far apart in its diagonal link.

\begin{lem}[Low rank complements]\label{lem:low-complements}
  If $x$ and $y$ are complements in a poset $P$ of rank at most~$3$
  then $v_x$ and $v_y$ are far apart in its diagonal link.
\end{lem}

\begin{proof}
  If $P$ has rank less than $3$ then its diagonal link has no edges
  and $v_x$ and $v_y$ are trivially far apart.  Thus we may assume
  that the rank of $P$ is $3$.  In this case, the diagonal link is a
  bipartite metric graph where every edge has length $\frac{\pi}{3}$
  and connects an element of rank~$1$ to an element of rank~$2$.  As a
  consequence $v_x$ and $v_y$ are not far apart iff their
  combinatorial distance is less than~$3$.  Distance~$1$ means $x <y$
  or $x > y$ and distance~$2$ means $x$ and $y$ are both rank $1$ with
  a common rank~$2$ upper bound or both rank $2$ with a common
  rank~$1$ lower bound.  All of these situations are excluded by the
  hypothesis that $x$ and $y$ are complements.
\end{proof}

This result has consequences for piecewise geodesic paths in the
$1$-skeleton of the diagonal link.

\begin{lem}[Low rank transitions]\label{lem:low-transitions}
  Let $P$ be a bounded graded poset of rank at most $4$.  If $x$, $y$
  and $z$ are distinct elements of $P$ such that $x$ and $z$ are
  complements in $P_{\zero y}$ or complements in $P_{y \one}$, then
  the edge path from $v_x$ to $v_y$ to $v_z$ in the diagonal link of
  $P$ is locally geodesic.
\end{lem}

\begin{proof}
  Suppose $x$ and $z$ are complements in $P_{\zero y}$; the other case
  is analogous.  By Lemma~\ref{lem:low-complements} $v_x$ and $v_z$ are
  far apart in $\lk(e_{\zero y}, \order{P_{\zero y}})$ since $P_{\zero
    y}$ is a poset of rank at most $3$.  Recall that the link of $v_y$
  in the diagonal link of $P$ is the spherical join $\lk(e_{\zero y},
  \order{P_{\zero y}}) \ast \lk(e_{y\one}, \order{P_{y\one}})$.  The
  fact that $v_x$ and $v_z$ are far apart in one factor means that
  $v_x$ and $v_z$ are far apart in the spherical join.  As a
  consequence, the path from $v_x$ to $v_y$ to $v_z$ in the diagonal
  link of $P$ is locally geodesic.
\end{proof}

Lemma~\ref{lem:low-transitions} quickly implies one-half of
Theorem~\ref{main:poset}.

\begin{thm}[Low rank spindles]\label{thm:low-spindle}
  If $P$ is a bounded graded poset of rank at most $4$, then global
  spindles in $P$ describe local geodesic loops in its diagonal link.
  As a consequence, if $P$ contains a short spindle, global or local,
  then the orthoscheme complex of $P$ is not $\cat(0)$.
\end{thm}

\begin{proof}
  The first assertion follows immediately from
  Lemma~\ref{lem:low-transitions}.  To see the second, suppose $P$
  contains a short spindle and restrict to the interval where it is
  global.  Since the spindle describes a short local geodesic loop in
  this local diagonal link, it is not $\cat(1)$ and by
  Theorem~\ref{thm:ortho-curv} the orthoscheme complex of $P$ is not
  $\cat(0)$.
\end{proof}

Having established that the existence of short spindles in low rank
posets prevent its orthoscheme complex from being $\cat(0)$, we pause
for a moment to clarify exactly which spindles in low rank posets are
short.  First note that every spindle of girth~$4$ is short (since
edges in the diagonal link have length less than $\frac{\pi}{2}$) and
they occur iff the poset is not a lattice
(Proposition~\ref{prop:spindle-bowtie}).  Thus we only need to consider
spindles in lattices.

\begin{prop}[Short spindles]\label{prop:short-spindles}
  If a bounded graded lattice $P$ of rank at most~$4$ contains a short
  spindle, then $P$ has rank~$4$, the spindle is a global spindle of
  girth $6$ and its elements alternate between two adjacent ranks.
\end{prop}

\begin{proof}
  After replacing $P$ by one its intervals if necessary we may assume
  that the spindle under consideration is a global spindle in $P$ and,
  since $P$ is a lattice, it must have girth at least~$6$.  If $P$ has
  rank~$3$ (lower ranks are too small to contain spindles) then all
  edges in the diagonal link of $P$ have length $\frac{\pi}{3}$ and
  the spindle is not short.  Thus $P$ has rank~$4$.  In rank~$4$ there
  are two possible edge lengths: the shorter edges connect adjacent
  ranks and have length $\arccos \left( \sqrt{\frac{1}{3}} \right)
  \cong .304\pi$ and the longer edges connect ranks $1$ and $3$ and
  have length $\arccos \left( \frac{1}{3} \right) \cong .392\pi$.
  (Exact values are calculated using Proposition~\ref{prop:edge-length}.)
  Since both lengths are more than $\frac{\pi}{4}$, spindles of girth
  $8$ or more are not short.  Finally, since one long and two short
  edges have total length exactly $\pi$ and spindles in this setting
  have to have an even number of longer edges, the only short spindles
  are those involving six short edges creating a zig-zag path between
  two adjacent ranks as shown in Figure~\ref{fig:bad}.
\end{proof}

\begin{figure}
\includegraphics{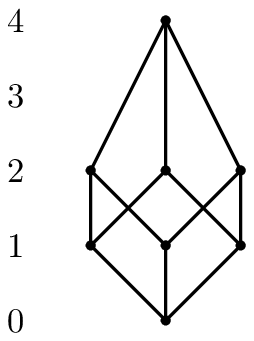}\hspace{1in}
\includegraphics{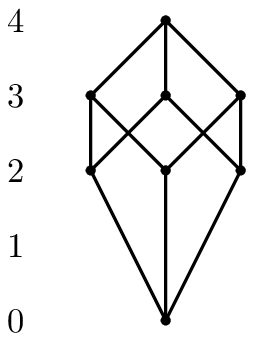}
\caption{Short spindles in rank~$4$ lattices.\label{fig:bad}}
\end{figure}

We should note that these low rank short spindles are closely related
to the empty triangles that arise when testing the curvature of cube
complexes.  Bowties cannot occur in the face poset of a cube complex
and the empty triangles that prevent cube complexes from being
$\cat(0)$ correspond to one of the short spindles of girth~$6$ just
described.

\begin{rem}[Short spindles and empty triangles]
  Let $K$ be cube complex that is not $\cat(0)$ and let $\sigma$ be a
  cell in $K$ whose link contains an empty triangle
  (Proposition~\ref{prop:cube-curvature}).  The zig-zag path shown on the
  lefthand side of Figure~\ref{fig:bad} corresponds to a portion of
  the face poset of the link of $\sigma$ in $K$.  The three elements
  in rank~$1$ and the three elements in rank~$2$ correspond to the
  three vertices and edges respectively of a triangle in the link of
  $\sigma$ and the absence of an element of rank~$3$ which caps off
  this zig-zag path corresponds to the fact that the triangle is
  empty.
\end{rem}

%%%
We now return to the proof of Theorem~\ref{main:poset}.  The second
basic result we need is that whenever the diagonal link of a low rank
poset contains a short local geodesic, that local geodesic can be
homotoped into the $1$-skeleton of the diagonal link without
increasing its length.  For posets of rank strictly less than $4$
there is nothing to prove and for rank~$4$ posets we appeal to earlier
work by Murray Elder and the second author \cite{ElMc02}.  The key
concept we need is that of a gallery.

\begin{defn}[Galleries]
  Given a local geodesic loop $\gamma$ in a \ps complex $K$ one can
  construct a new \ps complex $L$ called a \emph{gallery} such that
  the map $\gamma$ from a metric circle to $K$ factors through an
  embedding of the circle into $L$ and a cellular immersion of $L$
  into $K$.  The rough idea is to glue together copies of the cells
  through which $\gamma$ passes in $K$.  More specifically, every
  point in the path $\gamma$ is contained in a uniquely defined open
  simplex.  For this well-defined linear or cyclic sequence of open
  simplices, take a copy of the corresponding closed simplices and
  glue them together in the minimal way possible so that the result is
  a \ps complex that maps to $K$ and the curve $\gamma$ lifts though
  this map.  See \cite{ElMc02} for additional detail.
\end{defn}

\begin{defn}[Types of galleries]
  So long as the lengths of edges in $K$ are less than
  $\frac{\pi}{2}$, the gallery $L$ will be homotopy equivalent to a
  circle and the image of the circle in $L$ will have winding
  number~$1$.  Moreover, if $K$ is $2$-dimensional and the loop
  $\gamma$ does not pass through a vertex of $K$ then $L$ will be a
  $2$-manifold with boundary called either an \emph{annular gallery}
  or a \emph{m\"obius gallery} depending on its topology.  If $\gamma$
  does pass through a vertex of $K$ then $L$ is called a
  \emph{necklace gallery} and it can be broken up into segments called
  \emph{beads} corresponding to a portion of $\gamma$ starting at a
  vertex, ending at a vertex and not passing through a vertex in
  between.
\end{defn}

In \cite{ElMc02} a computer program was used to systematically
enumerate the finite list of possible galleries determined by a short
local geodesic in the vertex link of a \pe complex built out of
$\widetilde A_3$ Coxeter shapes.  (An $\widetilde A_n$ Coxeter shape
is a \pe tetrahedron whose vertex links are Coxeter shapes of type
$A_n$.  One general definition of these shapes is given in
Definition~\ref{def:columns}.)  The results of this enumeration are
listed in Figures~\ref{fig:a-cox}, \ref{fig:m-cox},
and~\ref{fig:v-cox} according to the following conventions.  The
triangles shown should be viewed as representing spherical triangles:
the angles that look like $\pi/2$ angles are in fact $\pi/2$ angles,
while the $\pi/4$ angles are meant to represent $\pi/3$ angles.  Thus,
in the third figure of Figure~\ref{fig:v-cox} both sides connecting
the specified end cells are actually geodesics as can be seen in a
more suggestive representation of the same configuration shown on the
righthand side of Figure~\ref{fig:lunes-E}.
The heavily shaded leftmost and rightmost edges in the configurations
shown in Figure~\ref{fig:a-cox} should be identified to produce annuli
and the heavily shaded leftmost and rightmost edges in the
configurations shown in Figure~\ref{fig:m-cox} should be identified
with a half-twist to produce m\"obius strips.  The three
configurations shown in Figure~\ref{fig:v-cox} are three of the beads
from which necklace galleries are formed.  They are labeled $C$, $D$
and $E$ since $A$ and $B$ are used to denote the short and long edges,
respectively, thought of as beads.  The following result was proved in
\cite{ElMc02}.

\begin{figure}
\begin{center}
\begin{tabular}{cc}
\begin{tabular}{c}\includegraphics{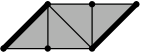}\end{tabular} &
\begin{tabular}{c}\includegraphics{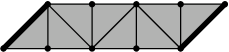}\end{tabular} 
\end{tabular}
\end{center}
\caption{Two annular galleries.\label{fig:a-cox}}
\end{figure}

\begin{figure}
\begin{center}
\begin{tabular}{cccc}
\begin{tabular}{c}\includegraphics{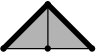}\end{tabular} &
\begin{tabular}{c}\includegraphics{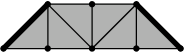}\end{tabular} &
\begin{tabular}{c}\includegraphics{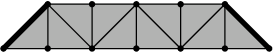}\end{tabular} &
\begin{tabular}{c}\includegraphics{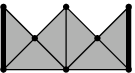}\end{tabular} 
\end{tabular}
\end{center}
\caption{Four m\"obius galleries.\label{fig:m-cox}}
\end{figure}

\begin{figure}
\begin{center}
\begin{tabular}{ccc}
\begin{tabular}{c}\includegraphics{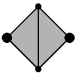}\end{tabular} &
\begin{tabular}{c}\includegraphics{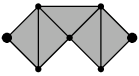}\end{tabular} &
\begin{tabular}{c}\includegraphics{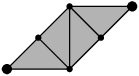}\end{tabular} \\
$C$ & $D$ & $E$\\
\end{tabular}
\end{center}
\caption{Three non-trivial beads.\label{fig:v-cox}}
\end{figure}

\begin{prop}[Short $A_3$ geodesics]\label{prop:a3}
  Let $K$ be a vertex link of a \pe complex built out of $\widetilde
  A_3$ Coxeter shapes.  If this \ps complex built out of $A_3$ Coxeter
  shapes is not quite $\cat(1)$ then it contains a short local
  geodesic loop $\gamma$ that determines a gallery $L$ that is either
  one of the two annular galleries listed in Figure~\ref{fig:a-cox},
  one of the four m\"obius galleries listed in Figure~\ref{fig:m-cox},
  or a necklace gallery formed by stringing together the short edge
  $A$, the long edge $B$, and the three nontrivial beads shown in
  Figure~\ref{fig:v-cox} in one of $26$ particular ways.  In
  particular, the $26$ necklace galleries that contain a short
  geodesic loop are described by the following sequences of beads:
  $A^2$, $A^4$, $A^6$, $A^2B$, $A^2B^2$, $A^2B^3$, $ABAC$, $A^2C$,
  $A^2C^2$, $A^2D$, $A^2E$, $A^4B$, $CA^4$, $B$, $B^2$, $B^3$, $B^4$,
  $B^5$, $BE$, $B^2E$, $C$, $C^2$, $C^3$, $CD$, $D$, and $E$.
\end{prop}

The relevence of Proposition~\ref{prop:a3} in the current setting is that
the diagonal link of a bounded graded rank~$4$ poset is a complex
built out of $A_3$ spherical simplices in a way that could have arisen
as a vertex link of a \pe complex built out of $\widetilde A_3$
shapes.  We will comment more on this connection in \S\ref{sec:artin}.
In particular, if the diagonal link of a bounded graded rank~$4$ poset
is not quite $\cat(1)$ then it contains a short unshrinkable local
geodesic loop that determines one of the $32$ specific galleries
listed above.

\begin{lem}[Loops and vertices]\label{lem:vertex}
  Let $K$ be the diagonal link of a bounded graded rank $4$ poset $P$.
  If $K$ is not quite $\cat(1)$ and $\gamma$ is a short unshrinkable
  locally geodesic loop in $K$, then $\gamma$ passes through a vertex
  of $K$.
\end{lem}

\begin{proof}
  Let $L$ be the gallery associated to $\gamma$.  If $L$ is an annular
  gallery then $\gamma$ is shrinkable via the analog of homotoping an
  equator through lines of latitude, contradicting our hypothesis.  If
  $L$ is a m\"obius gallery, then it is one of the four listed in
  Figure~\ref{fig:m-cox}.  Since $L$ immerses into the order complex
  $K$ of the diagonal link of a rank~$4$ poset we should be able to
  label each vertex of $L$ by the rank of the element of $P$ that
  corresponds to its image in $K$.  In particular, the three vertices
  of a triangle should receive three distinct numbers from the list
  $\{1,2,3\}$.  Once one triangle in $L$ is labeled, the remaining
  labels are forced and in each instance, the m\"obius strip cannot be
  consistently labelled.  As a result $L$ cannot be a m\"obius
  gallery.  The only remaining possibility is that $L$ is a necklace
  gallery and $\gamma$ passes through a vertex of $K$.
\end{proof}

\begin{figure}
\bt{cc}
\bt{c}\includegraphics{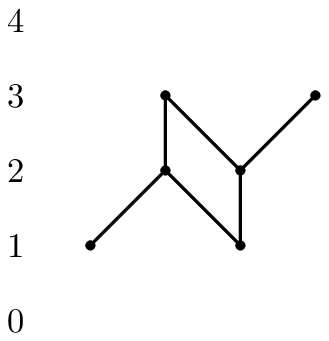}\et & \bt{c}\includegraphics{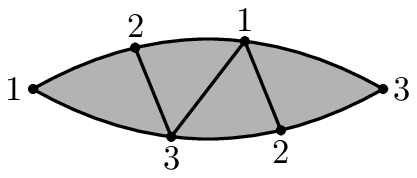}\et\\
\et
\caption{A poset configuration and the corresponding \ps configuration.\label{fig:lunes-E}}
\end{figure}

\begin{figure}
\bt{cc}
\bt{c}\includegraphics{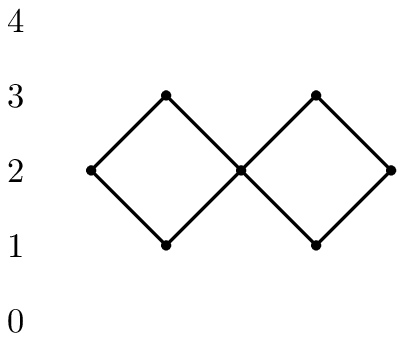}\et & \bt{c}\includegraphics{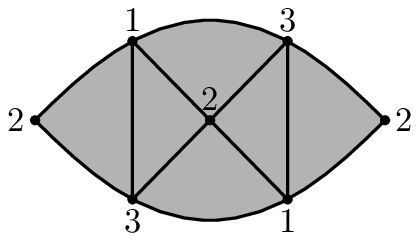}\et
\et
\caption{A poset configuration and the corresponding \ps configuration.\label{fig:lunes-D}}
\end{figure}

\begin{thm}[Restricting to the $1$-skeleton]\label{thm:1-skeleton}
  Let $P$ be a bounded graded poset of rank at most $4$.  If a local
  diagonal link of $P$ is not quite $\cat(1)$ then it contains a short
  local geodesic loop that remains in its $1$-skeleton.  In
  particular, when the orthoscheme complex of $P$ is not $\cat(0)$,
  $P$ contains a short spindle.
\end{thm}

\begin{proof}
  Since the diagonal link of an interval is $1$-dimensional when $P$
  has rank less than $4$, the first assertion is trivial in that case.
  Thus assume that $P$ has rank~$4$ and that we are considering the
  diagonal link $K = \lk(e_{\zero\one},\order{P})$.  Because $K$ is
  not quite $\cat(1)$, it is locally $\cat(1)$ but contains a short
  locally geodesic loop $\gamma$.  By Lemma~\ref{lem:bowditch} and
  Lemma~\ref{lem:vertex} we may also assume that $\gamma$ is both
  unshrinkable and that it is associated with a necklace gallery $L$.
  The necklace $L$ is a string of beads of type $A$, $B$, $C$, $D$ and
  $E$.  The bead of type $E$ is a lune, the portion of the geodesic it
  contains is of length $\pi$, and there is a length-preserving
  endpoint-preserving homotopy that moves this portion of $\gamma$
  into the boundary of $E$.  This has the effect of replacing $E$ with
  a sequence of edges ($AAB$ or $BAA$).  See Figure~\ref{fig:lunes-E}.
  Similarly, if $L$ contains a bead of type $D$, then $P$ contains a
  configuration that produces the lune shown in
  Figure~\ref{fig:lunes-D}.  The portion of the geodesic it contains
  is of length $\pi$, and there is a length-preserving
  endpoint-preserving homotopy that moves this portion of $\gamma$
  into the boundary of the lune.  This has the effect of replacing $D$
  with a sequence of edges $ABA$.  Thus we may assume that $L$
  contains no beads of type $D$ or $E$.

  Finally, suppose $L$ contains a bead of type $C$ and consider the
  bead immediately after it.  It cannot be of type $B$ since $C$ ends
  at a vertex of rank $2$ so it is either type $A$ or another bead of
  type $C$.  If type $A$ then we have all but one element of the
  configuration shown on the lefthand side of Figure~\ref{fig:lunes-E}
  and there is an additional triangle in $K$ that gives us a way to
  shorten $\gamma$, contradicting its unshrinkability.  On the other
  hand, if the next bead has type $C$ (and there are no obivous
  shortenings) then we have the configuration shown on the lefthand
  side of Figure~\ref{fig:lunes-D}. Thus there are triangles present
  in $K$ that enable us to perform a length-preserving endpoint
  preserving homotopy of this portion of $\gamma$ through beads of
  type $CC$ to a path in the $1$-skeleton passing through edges of
  type $ABA$.  In short, whenever $\gamma$ leaves the $1$-skeleton,
  there is a way to locally modify the path so that its length never
  increases and the new path passes through fewer $2$-cells.
  Iterating this process proves the first assertion and the second
  assertion follows from Proposition~\ref{prop:transitions}.
\end{proof}

Combining Theorem~\ref{thm:low-spindle} and
Theorem~\ref{thm:1-skeleton} establishes the following.

\setcounter{mainthm}{0}
\begin{mainthm}
  The orthoscheme complex of a bounded graded poset $P$ of rank at
  most~$4$ is $\cat(0)$ iff $P$ is a lattice with no short spindles.
\end{mainthm}

As a quick application note that Theorem~\ref{main:poset} implies that
every modular poset of rank at most $4$ has a $\cat(0)$ orthoscheme
complex with a $\cat(1)$ diagonal link.  In particular, the theorem
shows that the linear subspace poset $L_4(\F)$ has a $\cat(0)$
orthoscheme complex for every field $\F$ and its diagonal link, which
is a thick spherical building of type $A_3$, is $\cat(1)$.  That the
link is $\cat(1)$ is well-known.  See, for example, \cite{AbBr08} or
\cite{Ly05}.

%%%%%%%%%%%%%%%%%%%%%%%%%%%%%%%%%%%%%%%%%
\section{Artin groups}\label{sec:artin}
%%%%%%%%%%%%%%%%%%%%%%%%%%%%%%%%%%%%%%%%%

In this final section we first apply Theorem~\ref{main:poset} to a
poset closely associated with the $5$-string braid group and then, at
the end of the section, we extend the discussion to the other four
generator Artin groups of finite-type.  For the braid group, the
relevant poset is the lattice of noncrossing partitions.

\begin{figure}
  \includegraphics[scale=1]{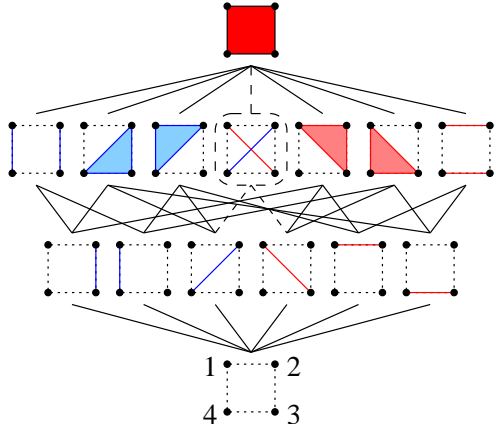}
  \caption{The figure shows the partition lattice $\Pi_4$ with the
    blocks indicated by their convex hulls.  If the portion surrounded
    by a dashed line is removed, the result is the noncrossing
    partition lattice $NC_4$.\label{fig:nc4}}
\end{figure}

\begin{defn}[Partitions and noncrossing partitions]
  Recall that a \emph{partition} of a set is a pairwise disjoint
  collection of subsets (called \emph{blocks}) whose union is the
  entire set. These partitions are naturally ordered by refinement,
  i.e. one partition is less than another is each block of the first
  is contained in some block of the second.  The resulting bounded
  graded lattice is called the \emph{partition lattice}.  Its maximal
  element has only one block, its minimal element has singleton blocks
  and the rank of an element is determined by the number of blocks it
  contains.  When the underlying set is $[n] =\{1,2,\ldots,n\}$ the
  partition lattice is denoted $\Pi_n$ and it has rank $n-1$. The
  rank~$3$ poset $\Pi_4$ is shown in Figure~\ref{fig:nc4}.  A
  \emph{noncrossing} partition is a partition of the vertices of a
  regular $n$-gon (consecutively labeled by the set $[n]$) so that the
  convex hulls of its blocks are pairwise disjoint.
  Figure~\ref{fig:ncross} shows the noncrossing partition
  $\{\{1,4,5\}, \{2,3\}, \{6,8\}, \{7\}\}$.  A partition such as
  $\{\{1,4,6\}, \{2,3\}, \{5,8\}, \{7\}\}$ would be crossing.  For
  $n=4$, the only difference between $\Pi_4$ and $NC_4$ is the
  partition $\{\{1,3\},\{2,4\}\}$ which is not noncrossing.  The
  subposet of noncrossing partitions is also a bounded graded rank~$n$
  lattice.  In addition, $NC_n$ is self-dual in the sense that there
  exists an order-reversing bijection from $NC_n$ to itself
  (\cite{Br01},\cite{Mc06}).
\end{defn}

\begin{figure}
\includegraphics[scale=1]{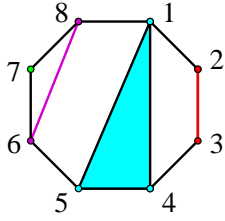}
\caption{A noncrossing partition of the set $[8]$.\label{fig:ncross}}
\end{figure}

The close connection between the braid groups and the noncrossing
partition lattice can briefly be described as follows.  There is a way
of pairwise identifying faces of the orthoscheme complex of $NC_n$ by
isometries so that the result is a one-vertex complex $X$ with a
contractible universal cover and the $n$-string braid group as its
fundamental group.  See \cite{Br01} for details.  Moreover, under the
orthoscheme metric, the metric space $X$ metrically splits as a direct
product of a compact \pe complex $Y$ and a circle of length
$\sqrt{n}$.  We should note, however, that this split is not visible
in the cell structure of $X$.  The splitting into a direct product is
easiest to see in the universal cover $\widetilde X \cong \widetilde Y
\times \R$ where the standard $n$-orthoschemes naturally fit together
into columns.

\begin{defn}[Columns]\label{def:columns}
  Fix $n\in \N$ and consider the following collection of points in
  $\R^n$.  For each integer $m$ write $m = qn+r$ where $q$ and $r$ are
  the unique integers with $0 \leq r < n$ and let $v_m$ denote the
  point $((q+1)^r,q^{n-r}) \in \R^n$ using the same shorthand notation
  as in the proof of Proposition~\ref{prop:edge-length}.  To
  illustrate, when $m=-22$ and $n=8$ then $q=-3$, $r=2$ and $v_{-22}
  =(-2^2,-3^6)$.  Note that the vector from $v_m$ to $v_{m+1}$ is a
  unit basis vector and that the particular unit basis vector is
  specified by the value of $r$.  In particular, any $n+1$ consecutive
  vertices of the bi-infinite sequence $(\ldots v_{-2}, v_{-1}, v_0,
  v_1, v_2 \ldots)$ are the vertices of a standard $n$-orthoscheme.
  It is easy to check that the standard $n$-orthoscheme
  $\ortho(v_0,v_1,\ldots,v_n)$ is defined by the inequalities $1 \geq
  x_1 \geq x_2 \geq \cdots \geq x_n \geq 0$ and that the union of the
  orthoschemes defined by $n+1$ consecutive vertices of this sequence
  is a convex shape defined by the inequalities $x_1 \geq x_2 \geq
  \cdots \geq x_n \geq x_1-1$.  We call this configuration of
  orthoschemes a \emph{column}.  Because these equalities are
  invariant under the addition of multiples of the vector $(1^n)$, the
  result is metrically a direct product of a $(n-1)$-dimensional shape
  with the real line.  It turns out that the cross-section
  perpendicular to the direction $(1^n)$ is a Euclidean polytope known
  as the \emph{Coxeter simplex of type $\widetilde A_{n-1}$} and that
  every vertex of this polytope has a link isometric to the diagonal
  link of the standard $n$-orthoscheme, namely, the convex spherical
  polytope of type $A_{n-1}$ that we called $\alpha_{n-1}$.  To
  illustrate, when $n=3$ the column just defined is a direct product
  of an $\widetilde A_2$ Euclidean polytope with $\R$ and the diagonal
  link of a $3$-orthoscheme is the spherical polytope of type $A_2$.
  In this case the \pe shape $\widetilde A_2$ is an equilateral
  triangle and the \ps shape $A_2$ is an arc of length
  $\frac{\pi}{3}$.
\end{defn}

Returning to the Eilenberg-MacLane space for the $n$-string braid
group, recall that a group is called a \emph{$\cat(0)$ group} if it
acts properly discontinuously cocompactly by isometries on a complete
$\cat(0)$ space.  For our purposes, the only fact we need is that the
fundamental group of any compact locally $\cat(0)$ \pe complex is a
$\cat(0)$ group since its action on the universal cover by deck
transformations has all of the necessary properties.  If $X$ is the
Eilenberg-MacLane space for the $n$-string braid group built from the
orthoscheme complex of the noncrossing partition lattice $NC_n$, then
as a metric space $X$ is a direct product of a circle of length
$\sqrt{n}$ and a \pe complex $Y$ built out of $\widetilde A_{n-1}$
shapes.  Moreover, the link of the unique vertex in $Y$ is isometric
to the diagonal link of $NC_n$.  As a consequence, if the orthoscheme
complex of the noncrossing partition lattice $NC_n$ is $\cat(0)$, then
$Y$ is locally $\cat(0)$, $X$ is locally $\cat(0)$ and the $n$-string
braid group is a $\cat(0)$ group.  In short we have the following
implication.

\begin{prop}[Partitions and braids]\label{prop:part-braid}
  If the orthoscheme complex of the noncrossing partition lattice
  $NC_n$ is $\cat(0)$, then the $n$-string braid group is a $\cat(0)$
  group.
\end{prop}

Since we (firmly) believe that the orthoscheme complex of $NC_n$ is
indeed $\cat(0)$ for every $n$, we formalize this assertion as a
conjecture.

\begin{conj}[Curvature of $NC_n$]\label{conj:ncn}
  For every $n$, the orthoscheme complex of the noncrossing partition
  lattice $NC_n$ is $\cat(0)$ and as a consequence, the braid groups
  are $\cat(0)$ groups.
\end{conj}

One reason we believe that Conjecture~\ref{conj:ncn} is true has to do
with the close connection between noncrossing partitions, partitions
and linear subspaces.

\begin{rem}[Partitions and linear subspaces]
  If $\F$ is a field and $\F^n$ is a vector space with a fixed
  coordinate system then there is a natural injective map from $\Pi_n$
  to $L_n(\F)$ that sends each partition to the subspace of vectors
  where the sum of the coordinates whose indices all belong to the
  same block sum to $0$.  For example, the partition $\{\{1,3,4\},
  \{2,5\}, \{6\}, \{7\}\}$ is sent to the $3$-dimensional subspace of
  $\F^7$ satisfying the equations $x_1+x_3+x_4 = 0$, $x_2+x_5=0$,
  $x_6=0$, and $x_7=0$.  The partition with one singleton blocks is
  sent to the $0$-dimensional subspace and the partition with one
  block is sent to the $(n-1)$-dimensional subspace where all
  coordinates sum to $0$.  Thus the map from $\Pi_n$ to $L_n(\F)$ can
  be restricted to a map from $\Pi_n$ to $L_{n-1}(\F)$ but at the cost
  of being harder to describe.  The relationship between the noncrossing
  partition lattice, the partition lattice and the lattice of linear
  subspaces is therefore $NC_n \subset \Pi_n \hookrightarrow
  L_{n-1}(\F)$.
\end{rem}

As we remarked earlier, the diagonal link of $L_{n-1}(\F)$ is a
$\cat(1)$ complex known as a thick spherical building.  Thus, the
inclusion just established means that the diagonal link of $NC_n$ is a
subcomplex of a thick spherical building.  For those familiar with the
structure of buildings, we note that a much stronger statement is
true.

\begin{prop}[Partitions and apartments]\label{prop:apartments}
  Every chain in $NC_n$ belongs to a boolean subposet.  As a
  consequence, for every field $\F$, the diagonal link of $NC_n$ is a
  union of apartments in the thick spherical building constructed as
  the diagonal link of $L_{n-1}(\F)$.
\end{prop}

\begin{proof}[Proof sketch]
  Since Proposition~\ref{prop:apartments} is not needed below, we shall not
  give a complete proof, but the rough idea goes as follows.  Every
  chain of noncrossing paritions can be extended to a maximal chain,
  and given any maximal chain, it is possible to systematically
  extract a planar spanning tree of the regular $n$-gon with edges
  labeled $1$ through $n-1$ such that the connected components of the
  graph containing only edges $1$ through $i$ are the blocks of the
  rank $i$ noncrossing partition in the chain.  Once such a labeled
  spanning tree has been found, the noncrossing partitions that arise
  from the connected components of the graph with an arbitrary subset
  of these edges form a boolean subposet of $NC_n$.  The second
  assertion follows since boolean subposets give rise to spheres in
  the diagonal link that are the apartments of the spherical building.
\end{proof}

The fact that the diagonal link of $NC_n$ is a union of apartments
inside a thick spherical building is circumstantial evidence that the
diagonal link is $\cat(1)$, the orthoscheme complex of $NC_n$ is
$\cat(0)$ and that the corresponding Eilenberg-MacLane space for the
$n$-string braid group is $\cat(0)$.  By Theorem~\ref{main:poset},
these conjectures are true for $n=5$.

\begin{prop}[Curvature of $NC_5$]\label{prop:nc4}
  The orthoscheme complex of $NC_5$ is $\cat(0)$ and, as a
  consequence, the $5$-string braid group is a $\cat(0)$ group.
\end{prop}

\begin{proof}
  Since the rank~$4$ poset $NC_5$ is known to be a lattice, by
  Theorem~\ref{main:poset} and Proposition~\ref{prop:short-spindles} we only
  need to check that $NC_5$ does not contain a global spindle of
  girth~$6$ whose elements alternate between adjacent ranks.
  Moreover, because $NC_5$ is self-dual, it is sufficient to rule out
  the configuration on the lefthand side of Figure~\ref{fig:bad}.
  Finally, if there were such a configuration, the three rank~$1$
  elements would correspond to noncrossing parititions each containing
  a single edge and the fact that they pairwise have rank~$2$ joins
  indicates that these edges are pairwise noncrossing.  But under
  these conditions, the join of all three elements will have rank~$3$
  contrary to the desired configuration.  Thus $NC_5$ has no short
  spindles, its orthoscheme complex is $\cat(0)$ and the $5$-string
  braid group is a $\cat(0)$ group.
\end{proof}

We should note that we originally proved that the $5$-string braid
group is a $\cat(0)$ group in a more direct fashion (unpublished)
shortly after the first author introduced his new Eilenberg-MacLane
spaces for the braid groups \cite{Br01}, a direct computation carried
out independently and contemporaneously by Daan Krammer (also
unpublished).
And finally, we indicate how the above analysis of the $5$-string
braid group can be extended to cover the other four-generator Artin
groups of finite-type.  The posets and complexes defined via the
symmetric group in \cite{Br01} were extended to the other finite
Coxeter groups in \cite{BrWa02}.  The first author's work with Colum
Watt produces bounded graded lattices with a uniform definition that
can be used to construct Eilenberg-MacLane spaces for groups called
Artin groups of finite-type.  We begin by roughly describing these
additional posets.

\begin{defn}[$W$-noncrossing partitions]
  Let $W$ be a finite Coxeter group with standard minimal generating
  set $S$ and let $T$ be the closure of $S$ under conjugacy.  The set
  $S$ is called a \emph{simple system} and $T$ is the set of all
  \emph{reflections}.  A \emph{Coxeter element} in $W$ is an element
  $\delta$ that is a product of the elements in $S$ in some order.
  For the finite Coxeter groups, the order chosen is irrelevant since
  the result is well-defined up to conjugacy.  The poset of
  \emph{$W$-noncrossing partitions} $NC_W$ is the derived from the
  minimum length factorizations of $\delta$ into elements of $T$ or
  equivalently, it represents an interval in the Cayley graph of $W$
  with respect to $T$ that starts at the identity and ends at
  $\delta$.  The name alludes to the fact that when $W$ is the
  symmetric group $\sym_n$, a Coxeter element is an $n$-cycle and the
  poset $NC_W$ is isomorphic to the lattice of noncrossing partitions
  previously defined.
\end{defn}

For each finite Coxeter group $W$, the poset $NC_W$ is a finite
bounded graded lattice whose rank~$n$ is the size of the standard
minimal generating set $S$ for $W$.  As was the case with the braid
groups, there is a one-vertex complex $X$ constructed by identifying
faces of the orthoscheme complex of $NC_W$.  This complex splits as a
metric direct product of a complex $Y$ constructed from $\widetilde
A_{n-1}$ shapes and a circle of length $\sqrt{n}$, and the universal
cover $\widetilde X$ decomposes into columns as before.  In
particular, $\widetilde X$ is isometric with $\widetilde Y \times \R$,
the link of the unique vertex of $Y$ is isometric to the diagonal link
of the orthoscheme complex of $NC_W$, and we have the following result
that generalizes Proposition~\ref{prop:part-braid}.

\begin{prop}[Partitions and Artin groups]\label{prop:part-artin}
  Let $W$ be a finite Coxter group and let $NC_W$ be its lattice of
  noncrossing partitions.  If the orthoscheme complex of $NC_W$ is
  $\cat(0)$ then the finite-type Artin group corresponding to $W$ is a
  $\cat(0)$ group.
\end{prop}

When $W$ has a standard minimal generating set of size $4$, the poset
$NC_W$ has rank~$4$ and Theorem~\ref{main:poset} can be applied as
above.  Each of the five possible posets (corresponding to the finite
Coxeter groups of type $A_4$, $B_4$, $D_4$, $F_4$ and $H_4$) is known
to be a lattice and thus by Proposition~\ref{prop:short-spindles} we
only need to check whether or not $NC_W$ contain a global spindle of
girth~$6$ whose elements alternate between adjacent ranks.  Moreover,
because $NC_W$ is self-dual (\cite{Mc06}), it is sufficient to search
for the configuration on the lefthand side of Figure~\ref{fig:bad}.
The second author wrote a short program in GAP to construct these
posets and to search for this particular configuration. The
noncrossing posets of type $A_4$ and $B_4$ contain no such
configurations but the noncrossing posets of type $D_4$, $F_4$ and
$H_4$ do contain such configurations.  By Theorem~\ref{main:poset} and
Proposition~\ref{prop:part-artin}, this establishes the following.

\setcounter{mainthm}{1}
\begin{mainthm}[Artin groups]
  Let $K$ be the Eilenberg-MacLane space for a four-generator Artin
  group of finite type built from the corresponding poset of
  $W$-noncrossing partitions and endowed with the orthoscheme metric.
  When the group is of type $A_4$ or $B_4$, the complex $K$ is
  $\cat(0)$ and the group is a $\cat(0)$ group.  When the group is of
  type $D_4$, $F_4$ or $H_4$, the complex $K$ is not $\cat(0)$.
\end{mainthm}

\bibliographystyle{plain}
\def\cprime{$'$} \def\cprime{$'$}

%%%%%%%%%%%%%%
\end{document}